\documentclass[12pt]{article}
\usepackage[T1]{fontenc}      
\RequirePackage[applemac]{inputenc} 

\scrollmode
\usepackage{color}
\usepackage{amsfonts}
\usepackage{amsmath}
\usepackage{amssymb}
\usepackage{graphicx}
\usepackage{mathptmx}
\usepackage{subfigure}
\usepackage{float}
\usepackage{parskip}
\RequirePackage[colorlinks, hyperindex, plainpages=false]{hyperref}
\usepackage{hyperref}
\usepackage{amsmath}

\def\E{{\mathbb E}}

\def\1{{\mathbf 1}}

\newtheorem{theorem}{Theorem}[section]
\newtheorem{lemma}[theorem]{Lemma}
\newtheorem{corollary}[theorem]{Corollary}
\newtheorem{proposition}[theorem]{Proposition}

\newtheorem{remark}[theorem]{Remark}

\newenvironment{proof}[1][Proof.]{\textbf{#1} }{\hfill $\blacksquare$}
\def\beq{\begin{equation}}
\def\eeq{\end{equation}}
\newcommand{\bei}{\begin{itemize}}
\newcommand{\eei}{\end{itemize}}
\newcommand{\ben}{\begin{enumerate}}
\newcommand{\een}{\end{enumerate}}
\newcommand{\beqn}{\begin{eqnarray}}
\newcommand{\beqnn}{\begin{eqnarray*}}
\newcommand{\eeqn}{\end{eqnarray}}
\newcommand{\eeqnn}{\end{eqnarray*}}
\newcommand{\brm}{\begin{rmk}}
\newcommand{\erm}{\end{rmk}}

\begin{document}

\title{Non-binary branching process and Non-Markovian exploration process.}
\author{
I.~Dram\'e  \footnote{  \scriptsize{LERSTAD, Université Gaston Berger, BP 234, Saint Louis, 
SENEGAL. ibrahima.drame@etu.univ-amu.fr
}}
\and
E.~Pardoux \setcounter{footnote}{6}\footnote{ \scriptsize {Aix-Marseille Université, CNRS,  Centrale Marseille, I2M, UMR 7373, 13453 Marseille, France. etienne.pardoux@univ-amu.fr}}
\and
A. B.~Sow \setcounter{footnote}{3}\footnote{ \scriptsize { LERSTAD, Université Gaston Berger, BP 234, Saint Louis, 
SENEGAL. ahmadou-bamba.sow@ugb.edu.sn}}
}


\maketitle

\begin{abstract}
We study the exploration (or height) process of a continuous time non-binary Galton-Watson random tree, in the subcritical, critical and supercritical cases. Thus we consider the branching process in continuous time $(Z_{t})_{t\geq 0}$, which describes the number of offspring alive at time $t$. We then renormalize our branching process and exploration process, and take the weak limit as the size of the population tends to infinity. Finally we deduce a Ray-Knight representation.
\end{abstract}

\vskip 3mm

\section{Introduction}
We consider a general continuous time branching process, describing a population where multiple births are allowed, unlike in the paper \cite{ba2012binary}. We first describe the exploration process, or height process, of the corresponding genealogical tree. We next study the convergence as the population size tends to infinity, of a properly rescaled version of it, towards a reflecting Brownian motion with drift. The difficulty is that we have to deal with a non Markovian exploration process. It had not been described so far in the literature. Taking the large population limit requests new arguments, in comparison with the binary branching situation studied in  \cite{ba2012binary}. 

We have carefully avoided to make any unnecessary assumption. In particular, we assume that the number of children  born at a given birth event has a finite second moment, and no higher order  moment. We hope to be able to treat in a near future the case without second moment, and study the limit of the genealogical trees in case where the limiting branching process is a continuous state branching process with jumps.

 In the supercritical case, as in \cite{ba2012binary} and \cite{delmas2006height}, we need to reflect the exploration process below an arbitrary level $\Gamma$, in order for this process to accumulate an arbitrary amount of local time at zero. This means killing the population at time $\Gamma$. It turns out that for taking the large population limit, reflection is also needed in the critical case. On the other hand, reflection is not required in the subcritical case. In order to be as concise as possible, we study the limit of the exploration process reflected below $\Gamma$ in the general case, and at the end show how the proof can be done without reflection in the subcritical case.  

The paper is organised as follows. Section 2 is devoted to the description of the height curves. In Section 3 we describe the relation between the laws of height processes and  non-binary continuous time Galton-Watson random trees. Finally, in Section 4 we present the results of convergence of the population process and the height process, in the limit of large populations.  In this paper a unique letter $C$ will denote a constant which may differ from line to line.
\section{Description of the exploration process}
In this section we will describe the exploration process of the non-binary tree associated to a continuous time branching process ${(Z_ {t})}_ {t \geq 0} $.
We fix $p>0$ and consider a continuous piecewise linear function $H$ from a subinterval of $\mathbb{R}_{+}$ into $\mathbb{R}_{+}$, which possesses the following properties : its slope is either $p$ or $-p$ ; it starts at time $t = 0$ from $0$ with the slope $p$ ; whenever $H(t) = 0$, $H_{-}^{\prime}(t) = -p$ and $H_{+}^{\prime}(t) = p$ ; $H$ is stopped at the time $T_m$ of its  $m$th return to zero, which is supposed to be finite. We will denote $\mathcal{H}_{p,m}$ the collection of all such functions. We will write $\mathcal{H}_{p}$ instead of $\mathcal{H}_{p,1}$. We now define a stochastic process whose trajectories belong to  $\mathcal{H}_{p}$ as follows. We choose the slopes of the piecewise linear process $H$ to be  $\pm$2  $($i.e $p$=2 $)$. We define the local time accumulated by $H$ at level $t$ up to time $s$: 
\begin{equation*}
  L_{s}(t) = \lim_{\varepsilon\mapsto 0}\frac{1}{\varepsilon} \int_{0}^{s}\mathbf{1}_{\{t\leq H_{r}< t+\varepsilon\}}dr.
\end{equation*}
$L_{s}(t)$ equals the number of pairs of branches of $H$ which cross the level $t$ between times $0$ and $s$. Let $\{V_{s}, \ s\geq 0\}$  be the càdlàg $\{-1,1\}$-valued process which  is such that, $s-$almost everywhere, ${dH_{s}}/{ds}=2V_{s}$. Let $\{{\Theta}_{k}, \ k \ge 1\}$ be a sequence of independent and identically distributed (i.i.d) random variables with values in $\mathbb{N}$. ${\Theta}_{k}$ will be the number of newborn at the $k$--th birth event, where this events are numbered in the order in which they are explored, see below. Let $\{P_{s}^{+}, \ s\geq 0\}$  (resp. $\{P_{s}^{-}, \ s\geq 0\}$)  be a Poisson process with intensity $\lambda$ (resp. $\mu$). We assume that the three processes $\{{\Theta}_{k}, \ k \ge 1\}$, $\{P_{s}^{+}, \ s\geq 0\}$ and  $\{P_{s}^{-}, \ s\geq 0\}$ are independent.\\
We are interested in the case where the number of children at each birth event is random, the exploration process  $H$ is defined from the process $V$ by the following equation :
\begin{align*}
\frac{dH_{s}}{ds}&=2V_{s}, {~~}  H_{0}=0, {~~} V_{0}=1\nonumber\\
V_{s}&=1+ 2\int_{0}^{s}\mathbf{1}_{\{V_{r^{-}}=-1\}}dP_{r}^{+} -2\int_{0}^{s}\mathbf{1}_{\{V_{r^{-}}=+1\}}dP_{r}^{-}+  2 (L_{s}(0)-L_{0^{+}}(0)) \nonumber\\
&+ 2\sum_{k>0, S_{k}^{+}\leq s}(L_{s}(H_{S_{k}^{+}})-L_{S_{k}^{+}}(H_{S_{k}^{+}}))\wedge ({\Theta}_{k}-1),
\end{align*}
where the $(S_{k}^{+}, k\geq 1)$ are the successive jump times of the process  $$\widetilde{P}_{s}^{+}=\int_{0}^{s}\mathbf{1}_{\{V_{r^{-}}=-1\}}dP_{r}^{+},$$ and where 
$L_{s}(t)$  denotes  the number of visits to level $t$ by the process $H$ up to time $s$,  $H_{0}=0$ and  $V_{0}=1$. For any $k>0$, ${\Theta}_{k}-{1}$ denotes the number of reflections of $H$ above the level $H_{S_{k}^{+}}$. Recall that ${\Theta}_{k}$ denotes the number of brothers and sisters born at the $k$th time of birth.
We define
\begin{equation*}
a=\sum_{\ell\ge1} \ell  \mathbb{P}({\Theta}_{1}=\ell)   \quad   \mbox{and} \quad \zeta^{2}=\sum_{\ell\ge1} (\ell-a)^{2} \mathbb{P}({\Theta}_{1}=\ell),
\end{equation*}
respectively the expectation and the variance of the number of births at each birth event.  In this work, we assume that these two quantities are finite. Let $\pi$ denote the common law of the random variables $\{{\Theta}_{k},\ k \ge 1\}$. We write $\Upsilon$ for the subcollection $\pi, \lambda, \mu$ (i.e $\Upsilon=\{\pi, \lambda, \mu\}$) and we denote by $\mathbb{P}_{\Upsilon}$ the law of the random element  $(H_s, \ s\ge 0)$ of $\mathcal{H}_{2}$. The random trajectory which we have constructed is an excursion above zero (see Figure 1 (B) ). We similarly define a law on $\mathcal{H}_{2,m}$ as  the concatenation of $m$ i.i.d such excursions. 
\begin{figure}[H]
\centering
\includegraphics[width=3.5in]{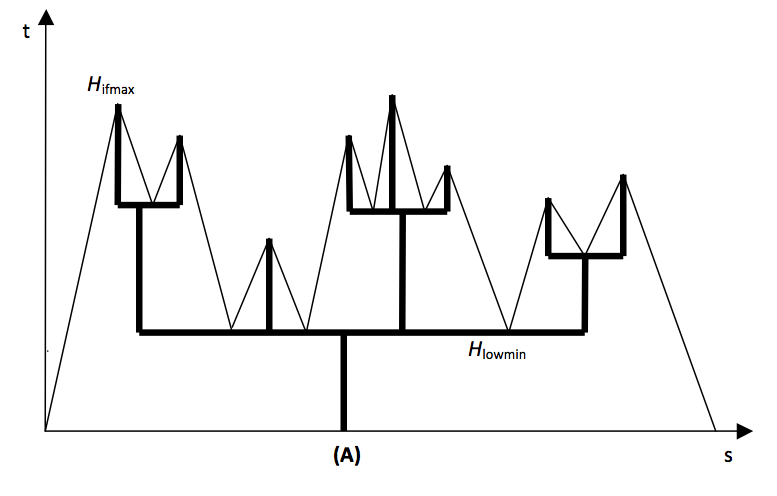}
\includegraphics[width=3.5in]{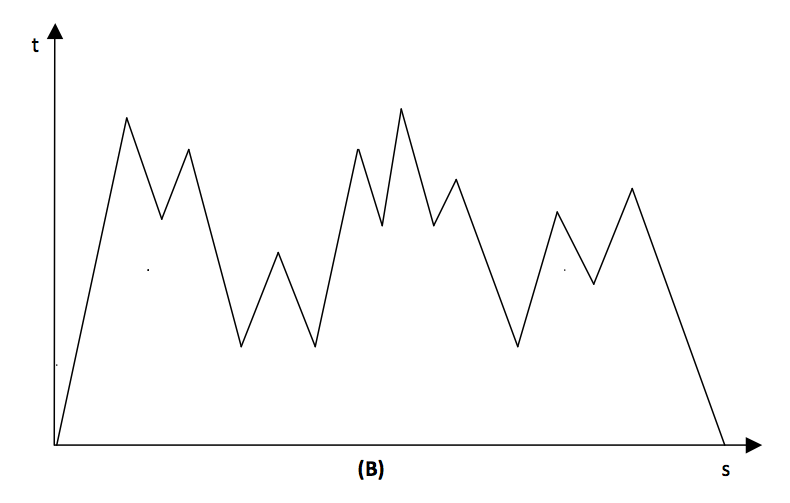}
\caption{ (A) The non-binary tree and its associated exploration process. (B) The exploration process. The $t$-axis is real time as well as exploration height, the $s$-axis is exploration time}
\end{figure}
\par We denote by $\Im$ the set of finite rooted non binary trees which are defined as follows. An ancestor is born at time $0$. Until she eventually dies, she produces a random number of offsprings. The same happens to each of her offspring, the offspring of her offspring, etc., until eventually the population dies out (assuming for simplicity that we are in the subcritical case). We denote by $\Im_{m}$ the set of forests which are the union of $m$ elements of $\Im$.  There is a well-known bijection between trees and exploration processes. Under the curve representing an element $H$ $\in$  $\mathcal{H}_{p}$, we can draw a tree as follows. The height $H_{ifmax}$ of the leftmost local maximum of $H$ is the lifetime of the ancestor and the height $H_{lowmin}$ of the lowest nonzero local minimum is the birth time of the first offsprings of the ancestor. If there is no  such local minimum, the ancestor dies childless. We draw a horizontal line at level $H_{lowmin}$. $H$ makes ${\Theta}_{1}+1$ excursions above $H_{lowmin}$. The leftmost excursion is used to represent the fate of the ancestor and of the rest of her progeny, excluding the first offsprings and their progeny. The ${\Theta}_{1}$ others excursions describe the fate of the first offsprings and their progeny. If there is no other local minimum of $H$ to the left or to the right of the first explored one, then there is no further birth: We draw a vertical line up to the unique local maximum, whose height is a death time. Continuing until there is no further local minimum/maximum to explore, this procedure defines a bijection $\Phi_{p}$ from $\mathcal{H}_{p}$ into $\Im$ (see Figure 1). Repeating the same construction $m$ times, we extend $\Phi_{p}$  to a bijection between  $\mathcal{H}_{p,m}$ and $\Im_{m}$. Note that the horizontal distances between the vertical branches in the tree representation of the exploration process are arbitrary. See Figure 1(A). \\ To the exploration process $H$, we associate the continuous-time Galton-Watson tree (which is a random element of $\Im$) with the same law $\pi$ and the same pair of parameters $(\mu,\lambda)$ as follows. The lifetime of each individual is exponential with parameter $\mu$. The birth events come according to a Poisson process with rate $\lambda$ and each time of birth, there is a random number of offsprings with law $\pi$. The behaviors of the various individuals are i.i.d. This defines a probability measures $\mathbb{Q}_{\Upsilon}$ on $\Im$. We use the same notation to denote the law on  $\Im_{m}$ of $m$ i.i.d random trees with $\mathbb{Q}_{\Upsilon}$ as their common law.

In the supercritical case, the exploration process defined above does not come back to $0$ a.s. To overcome this difficulty, we use a trick which is due to Delmas  \cite{delmas2006height}, and reflect the process $H^{\Gamma}$ below an arbitrary level $\Gamma>0$ (which amounts to kill the whole population at time $\Gamma$). The height process $H^{\Gamma}=\{H_s^{\Gamma}, \ s\ge 0\}$ reflected below $\Gamma$ is defined as above, with the addition of the rule that whenever the process reaches the level $\Gamma$, it stops and starts immediately going down with slope $-p$ for an exponential duration of time with expectation ${1}/{\lambda}$.  Again, the process stops when first going back to $0$. The reflected process $H^{\Gamma}$ comes back to zero almost surely, see \cite{ba2012binary} . In this case, the exploration process  $H^{\Gamma}$ is defined from the process $V$ by the following equation :
\begin{align*}
\frac{dH_{s}^{\Gamma}}{ds}&=2V_{s}, {~~}  H_{0}^{\Gamma}=0, {~~} V_{0}=1\nonumber\\
V_{s}&=1+ 2\int_{0}^{s}\mathbf{1}_{\{V_{r^{-}}=-1\}}dP_{r}^{+} -2\int_{0}^{s}\mathbf{1}_{\{V_{r^{-}}=+1\}}dP_{r}^{-}+  2 (L_{s}^{\Gamma}(0)-L_{0^{+}}^{\Gamma}(0)) \nonumber\\
&-2  L_{s}^{\Gamma}(\Gamma^{-})  + 2\sum_{k>0, S_{k}^{+}\leq s}(L_{s}^{\Gamma}(H_{S_{k}^{+}}^{\Gamma})-L_{S_{k}^{+}}^{\Gamma}(H_{S_{k}^{+}}^{\Gamma}))\wedge ({\Theta}_{k}-1),
\end{align*}
where $L_{s}^{\Gamma}(t)$  denotes  the number of visits to level $t$ by the process $H^{\Gamma}$ up to time $s$. \\For each $\Gamma$ $\in$ $(0, +\infty)$, and any subcollection $\Upsilon=\{\pi, \lambda, \mu\}$, denotes by $\mathbb{P}_{\Upsilon, \Gamma}$ the law of the process $H^{\Gamma}$. Define $\mathbb{Q}_{\Upsilon, \Gamma}$ to be the law of the $(\pi, \lambda, \mu)$ Galton-Watson tree, killed at time $t=\Gamma$ (i.e. all individuals alive at time $ \Gamma^{-}$ are killed at time $ \Gamma$). \par This reflection below an arbitrary level $\Gamma$ will be necessary both in the critical and in the supercritical cases. For our large population convergence result, we shall treat the subcritical case without this reflection in the last subsection of the paper.  
\section{Correspondence of laws}
\label{foo}
The aim of this section is to prove that for any subcollection $\Upsilon=\{\pi, \lambda, \mu\}$ and $\Gamma$  $\in$ $(0, +\infty)$ [including possibly $\Gamma= +\infty$ in the subcritical case],  $\mathbb{P}_{\Upsilon,\Gamma}{~}\Phi_{p}^{-1}=\mathbb{Q}_{\Upsilon, \Gamma}$. Let us state some basic results for homogeneous Poisson processes, which will be useful in the sequel.
\subsection{Preliminary results}
Let $(T_k)_{k\geq0}$ be a Poisson point process on ${\mathbb{R}}_{+}$ with intensity $\lambda$. This means that $T_0=0$ and $(T_{k+1}-T_k, k\geq0)$ are i.i.d. exponential random variables with mean $1/\lambda$. Let $(N_t, t\geq0)$ be the counting process associated with $T$, that is, for all $t\geq0$, 
\begin{equation*}
N_t = \sup\{k\geq0, T_k\leq t\}. 
\end{equation*}
The following result is well known and elementary.
\begin{lemma}\label{LAA}
Let $M$ be a nonnegative random variable independent of $T$, and define 
\begin{equation*}
R_M = \sup_{k\geq0}\{T_k, T_k\leq M\}. 
\end{equation*}
Then $M-R_M \stackrel{(d)}{=} V\wedge M$, where $V$ and $M$ are independent, and $V$ has an exponential distribution with mean $1/\lambda$. \par Morever,  on the event $\{R_M > s\}$, the conditional law of $N_{{R_M^{-}}}-N_s$ given $R_M$ is Poisson with parameter $\lambda(R_M - s)$.
\end{lemma}
\noindent In addition, we have the following result, which is Lemma 3.2 in \cite{ba2012binary}.
\begin{lemma}\label{LAB}
Let  $T=(T_k)_{k\geq0}$ be a Poisson point process on ${\mathbb{R}}_{+}$ with intensity $\lambda$, and let $M$ be a positive random variable which is independent of $T$. Consider  the integer-valued random variable $K$ such that $T_K=R_M$ and a second Poisson point process  ${T}^{'}=({T}_{k}^{'})_{k\geq0}$ with intensity $\lambda$, which is jointly independent of the first and of $M$. Then   $\bar{T}=(\bar{T}_{k})_{k\geq0}$, defined by 
\begin{equation*}
\bar{T}_{k}=
\left\{
    \begin{array}{ll}
   {T}_{k} \quad \quad \quad \quad  \quad if \quad k< K, &\\\\ {T}_{K}+{T}_{k-K+1}   \quad if \quad k\leq K,
    
&
           \end{array}
           \right.
\end{equation*}
is a Poisson point process on $\mathbb{R}_{+}$ with intensity $\lambda$, which is independent of $R_M$.
\end{lemma}

\subsection{Basic theorem}
\par Let  $\{{U}_{\ell}, \ \ell \ge 1\}$ and $\{{V}_{\ell}, \ \ell \ge 1\}$ be two sequences of independent and identically distributed (i.i.d) exponential random variables with means $1/ \mu $ and $1/ \lambda$, respectively.  Let  $(T_{k}^{\ell}, k\geq1, \ell \geq 1)$ be a  family of mutually independent Poisson processes with intensity $\lambda$. In the same way as Section 2, we introduce the sequence $\{{\Theta}_{\ell}, \ \ell \ge 1\}$ of i.i.d random variables with law $\pi$. We assume that the processes $(T_{k}^{\ell}, k\geq1, \ell \geq 1)$,  $\{{U}_{\ell}, \ \ell \ge 1\}$ and $\{{\Theta}_{\ell}, \ \ell \ge 1\}$ are mutually independent. The link between the random variables ${V}_{\ell}$ and the process $(T_{k}^{\ell}, \ k\geq1, \ell \geq 1)$ is given by Lemma  \ref{LAA}.
\par We are now in a position to prove the next theorem, which states a one-to-one correspondence between the tree associated with the exploration process $H^{\Gamma}$ defined in Section 2, and a continuous-time non binary Galton-Watson tree with the same law $\pi$ and the same pair of parameters $(\mu,\lambda)$, killed at time $\Gamma$. 
\begin{theorem}
For any $\Upsilon$ and $\Gamma$ $\in$ $(0, +\infty)$ (including the case $\Gamma=+\infty$ when  $\lambda< \mu$), the following representation holds
\begin{equation*}
\mathbb{P}_{\Upsilon,\Gamma}{~}\Phi_{p}^{-1}=\mathbb{Q}_{\Upsilon,\Gamma}.
\end{equation*}
\end{theorem}

\begin{proof}
The individuals making up the population represented by the tree whose law is $\mathbb{Q}_{\Upsilon}$ are labeled $\ell$ = 1, 2,..., with individual 1 born at time $m_0=1$ corresponding to the ancestor of the whole family. The subsequent individuals will be identified below. We will show that this tree is 'explored' by a process whose law is precisely $\mathbb{P}_{\Upsilon}$.  $U_{\ell}$ will be the lifetime of individual $\ell$.  For any $\ell \ge 1$,  the birth times of the offsprings of individual $\ell$ are  $\{T^{\ell}_{k}, \ 1 \le k \le K_\ell \}$, where $K_\ell= \sup\{ k, \ T^{\ell}_{k}\le U_{\ell} \} $.\\If $\ell$ is not the sister of an already explored individual born at the same time, i.e. if \\$m_{\ell-1}$ $\notin$ $\left\{m_1, m_2,...,m_{\ell-2}\right\}$, then we define $\Delta_{\ell}= {\Theta}_{\ell}-1$  the number of sisters of individual $\ell$ born at time $m_{\ell-1}$. If $\ell$ is the sister of individual $j< \ell$, then we let $\Delta_{\ell}= \Delta_{j}-1$.

$\textit{Step} 1$. We start from the initial time $t=0$ and climb up to level $M_1= U_1\wedge \Gamma$. We go down from $M_1$ until we find the most recent point of the Poisson process $(T_{k}^{1})$ ( recall that this process gives the birth times of the offsprings of individual 1). By Lemma \ref{LAA}, we have descended a height $V_1 \wedge M_1$. We hence reach the level  $m_1= (M_1- V_1) \vee 0$. If $m_1=0$, we stop, else we turn to the next step.  

$\textit{Step} 2$. We assign label 2 to the first offspring of the last birth event of offsprings of individual 1, born at time $m_1$ and we let  $\Delta_{2}= {\Theta}_{2}-1$ denote the number of unexplored sisters of individual 2 born at the same time her. Let us define 
$(\bar{T}_{k}^2)$ by 
\begin{equation*}
\bar{T}_{k}^2=
\left\{
    \begin{array}{ll}
   {T}_{k}^1 \quad \quad \quad \quad  \quad \mbox{if} \quad k< K_1, &\\\\ {T}_{K_1}^1+{T}_{k-K_1+1}^2   \quad \mbox{otherwise};
    
&
           \end{array}
           \right.
\end{equation*}
where $K_1$ is such that ${T}_{K_1}^1=m_1$. \\Thanks to Lemma \ref{LAB}, $(\bar{T}_{k}^2)$ is a Poisson process with intensity $\lambda$ on ${\mathbb{R}}_{+}$, which is independent of $m_1$ and in fact also of $(U_1, V_1)$. \\Starting from $m_1$, the exploration process climbs up to level $M_2=( m_1+U_2)\wedge \Gamma$. Starting from level $M_2$, if $\Delta_{2}=0$, we go down a height $M_2 \wedge V_2$, to find the most recent point of the Poisson process $(\bar{T}_{k}^2)$. At this time we are at level $m_2= (M_2- V_2) \vee 0$. If however $\Delta_{2}\ge 1$, we go down a height $ V_2 \wedge (M_2-m_1)$ and in this case we are at level $m_2= (M_2- V_2) \vee m_1$. If $m_1=m_2$, we change the value of $\Delta_{2}$, and let it be equal to $\Delta_{2}-1$. If $m_2=0$, we stop. Otherwise we continue. \\ Suppose we have made $\ell-1$ steps and $m_{\ell-1}\ge  0$, $\ell \ge 3$.  

$\textit{Step} \ell$. We start from $m_{\ell-1}$ which is the birth time of individual $\ell$. Note that by then for all $2\leqslant j \leqslant\ell$, $\Delta_j$ is the number of sisters of individual $j$ who still remain to be explored. We now define 
\begin{equation*}
\bar{T}_{k}^{\ell}=
\left\{
    \begin{array}{ll}
   \bar{T}_{k}^{\ell-1} \quad \quad \quad \quad  \quad \mbox{if} \quad k< K_{\ell-1}, &\\\\\bar{T}_{K_{\ell-1}}^{\ell-1}+\bar{T}_{k-K_{\ell-1}+1}^{\ell}  \quad \mbox{otherwise};
    
&
           \end{array}
           \right.
\end{equation*}
Then $(\bar{T}_{k}^{\ell})$ is a Poisson point process with intensity $\lambda$ on ${\mathbb{R}}_{+}$ and is independent of $(m_1, M_1, \cdot \cdot \cdot, m_{\ell-1}, M_{\ell-1})$. \\Starting from $m_{\ell-1}$, the height process climbs up to level $M_{\ell}= (m_{\ell-1}+U_{\ell})\wedge \Gamma$, which is the time of death of individual $\ell$. We set 
\begin{equation*}
{\ell}^\ast=
\left\{
    \begin{array}{ll}
   \sup\{2\leqslant j \leqslant\ell, \ \Delta_j>0\}, \quad \mbox{if} \quad \inf_{{2\leqslant j \leqslant\ell}} \ \Delta_j>0, &\\\\\ 1,  \quad \mbox{otherwise}.
    
&
           \end{array}
           \right.
\end{equation*}
Note that, if $\Delta_{\ell}>0$,  $\ell^\ast=\ell$. Coming down from level $M_{\ell}$, if ${\ell}^\ast=1$, we wait a time $ V_{\ell} \wedge M_{\ell}$, to find the most recent point of the Poisson process $(\bar{T}_{k}^{\ell})$. At this time we are at level  $m_{\ell}=(M_{\ell}-V_{\ell}) \vee0$. If however ${\ell}^\ast \ge 2$  we go down a height $ V_\ell \wedge (M_\ell-m_{\ell^\ast-1})$ and in this case we are at level $m_\ell= (M_\ell- V_\ell) \vee m_{\ell^\ast-1}$. If $m_\ell= m_{\ell^\ast-1}$, we change the value of $\Delta_{{\ell}^\ast}$, and let it be equal to $\Delta_{{\ell}^\ast}-1$. See Figure 2.

\begin{figure}[H]
\centering
\includegraphics[width=3.0in]{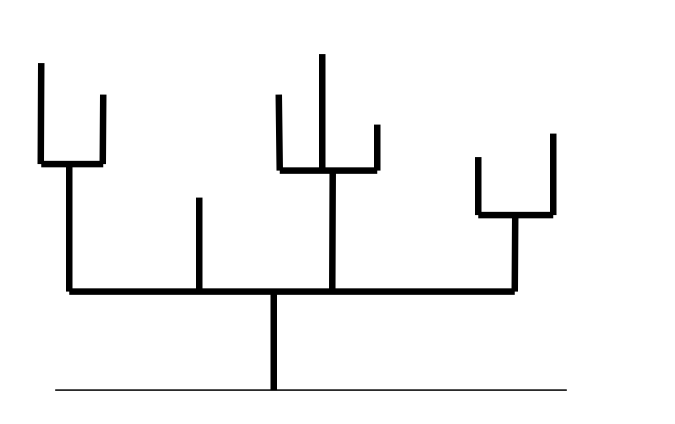}
\includegraphics[width=3.6in]{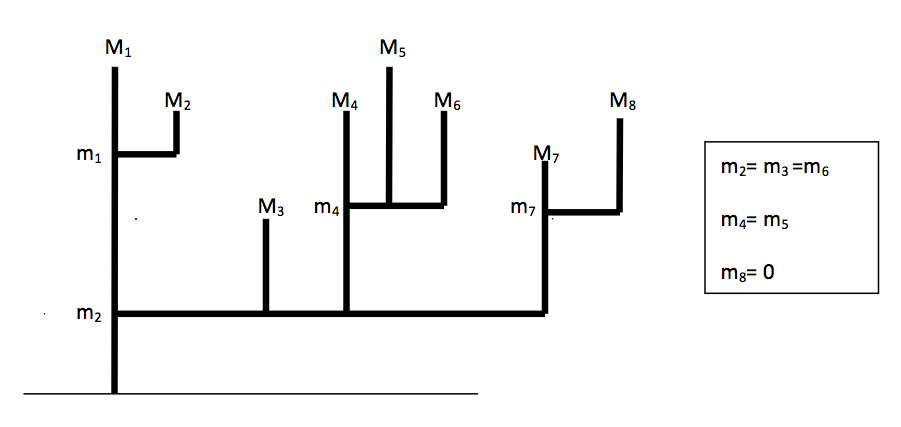}
\caption{ Two equivalent ways of representing a non binary tree.}
\end{figure}
Since either we have a reflection at level $\Gamma$ or we are in the subcritical case, zero is reached a.s. after a finite number of iterations. It is clear that the random variables $M_i$ and $m_i$ fully determine the law  $\mathbb{Q}_{\Upsilon,\Gamma}$ of the non binary tree killed at time $t=\Gamma$ and they have both the same joint distribution as the levels of the successive local minima and maxima of the process $H^{\Gamma}$ under  $\mathbb{P}_{\Upsilon,\Gamma}$, see, e.g. \cite{ba2012binary}.
\end{proof}
\section{Weak convergence} 
\subsection{Renormalization} 
Let $x>0$ be arbitrary, and let $N\geqslant1$ be an integer which will eventually go to infinity. Let $(Z_{t}^{N,x})_{t\geq 0}$ denote the branching process which describes the number of offsprings at time $t$ of $[Nx]$ ancestors in the population with birth rate  $\lambda_{N}= {N\sigma^2}/{2a}+{\alpha}/{a}$  and death rate $\mu_{N}= {N\sigma^2}/{2}+{\beta}$, where $\alpha \ge 0$, $\beta \geq 0$ and $\sigma > 0$. In this population, the number of children at each birth event is a random variable $\Theta$ that has the same law as ${\Theta}_{1}$. We now define the rescaled continuous time process
\begin{equation}\label{XN}
X_{t}^{N,x}:=N^{-1}Z_{t}^{N,x}.
\end{equation}
In particular, we have
\begin{align*}
 X_{0}^{N,x}= {[Nx]}/{N}\longrightarrow x \quad  as \quad   N \rightarrow +\infty. 
\end{align*}
Let $H^{N,\Gamma}$ be the exploration process associated to $\{Z_{t}^{N,x},  \ 0\leq t\leq \Gamma\}$ defined in the same way as previously, but with slopes $\pm 2N$, and where $\lambda$, $\mu$ are replaced by $\lambda_{N}$ and $\mu_{N}$ to be specified below.  We define also $L_{s}^{N,\Gamma}(t)$, the local time accumulated by $H^{N,\Gamma}$ at level $t$ up to time $s$, as
\begin{equation}\label{TL}
  L_{s}^{N,\Gamma}(t) = \frac{4}{\kappa^{2}\delta} \lim_{\varepsilon\mapsto 0}\frac{1}{\varepsilon} \int_{0}^{s}\mathbf{1}_{\{t\leq H_{r}^{N,\Gamma}< t+\varepsilon\}}dr,
\end{equation}
where $\delta=\frac{1}{2a}(a+a^2+\zeta^2)$ and $\kappa^2=\sigma^2\delta$.
The motivation of the factor  $4/\kappa^{2}\delta$ will be clear after we have taken the limit as $N \rightarrow  +\infty$. $L_{s}^{N,\Gamma}(t)$ equals $4 / N\kappa^{2}\delta$ times the number of pairs of $t$-crossings of $H^{N,\Gamma}$  between times $0$ and $s$. Note that this process is neither right- nor left-continuous as a function of $s$. \\Let $\tau_{x}^{N,\Gamma}$ be the time to explore the forest $\Im_{N,x}$. We have that
\begin{equation*}
\tau_{x}^{N,\Gamma}=\inf \left\{s>0: \   L_{s}^{N,\Gamma}(0)>\frac{4}{\kappa^{2}\delta} \frac{[Nx]}{N}\right\}.
\end{equation*}
We define, for all $N\geq1$, the projective limit $\{ \mathcal{L}_{x}^{N}(t), \ t\geq 0, \ x>0\}$, which is such that, for each $\Gamma$, 
\begin{equation*}
 \left\{  \mathcal{L}_{x}^{N}(t), \  0\leq t\leq \Gamma, \ x>0  \right\}  \stackrel{(d)}{=} \left\{ L_{\tau_{x}^{N,\Gamma}}^{N,\Gamma}(t), \ 0\leq t\leq \Gamma, \ x>0 \right\}.
\end{equation*}
Proposition 3.1 in \cite{ba2012binary} translates as (note that the factor $N^{-1}$ in the definition of $Z_{t}^{N,x}$ matches the slopes $\pm2N$ of $H^{N,\Gamma}$, which introduces a factor  $N^{-1}$ in the local times defined by \eqref{TL})  
\begin{lemma}
We have the identity in law
\begin{equation*}
\left\{  \mathcal{L}_{x}^{N}(t), \  t\ge0, \ x>0\right\} \stackrel{(d)}{=} \left\{\frac{4}{\kappa^{2}\delta} X_{t}^{N,x}, \ t\ge0, \ x>0  \right\}.
\end{equation*}
\end{lemma}
\subsection{Tightness criteria in  \texorpdfstring{ $C([0, +\infty))$}{Lg} and   \texorpdfstring{$D([0, +\infty))$}{Lg}} 
We shall start with the basic tightness criterion for random processes on the space of continuous functions $C([0, +\infty))$. For $T>0$, we define $w_{x, T}(.)$ the modulus of continuity of $x$ $\in$ $C([0, +\infty))$ on the interval $[0,T]$ by 
\begin{equation}\label{ModulC}
w_{x,T}(\rho)= w_{T}(x,\rho)= \sup_{|s-t|\leqslant \rho, \ s,t\leqslant T} |x(s)-x(t)| , \quad \rho>0. 
\end{equation}
Consider now a sequence $\{X^{n}, n\geq1 \}$ of random processes with trajectories in $C([0, +\infty))$. The following proposition follows from Theorem 7.3 in  \cite{Billingsley:2009rz}. 
\begin{proposition}\label{AZERO}
The sequence ${\{X^{n}, n\geq1 \}}$ is tight in $C([0,\infty))$ iff the two following conditions hold

$(i)$ for each $\eta \geq0$, there exist an $b$ and an $n_0$ such that $$\mathbb{P}(|X^{n}(0)|\ge b)\leqslant \eta, \quad n\ge n_0.$$
$(ii)$ for each $\epsilon \geq0$,  $T>0$, $$\lim_{\rho \rightarrow0} \limsup_{n\rightarrow \infty} \mathbb{P}(w_{T}(X^{n},\rho)\geq \epsilon)=0.$$
\end{proposition}
\begin{corollary}\label{AZEOR}
Condition $(ii)$ of Proposition \ref{AZERO} hold if, for each positive $\epsilon$ and $\eta$, there exist a $\rho$, $\rho>1$, and integer $n_0$ such that for each $T>0$  $$\frac{1}{\rho} \mathbb{P}\left(\sup_{t\leq s\leq t+\rho} |X^{n}(s)-X^{n}(t)|\ge \epsilon \right)\leqslant \eta, \quad  for \ every \ n\ge n_0 \ and \ 0\leqslant  t\leqslant  T.$$
\end{corollary}
Let us present a sufficient condition for tightness which will be useful below. Consider a sequence ${\{X_{t}^{n}, t\geq0 \}}_{n\geq1}$ of one-dimensional semi-martingales, which is such that for each $n\geq1$,
\begin{eqnarray*}
X_{t}^{n}= X_{0}^{n}+ \int_{0}^{t}{\varphi}_{n}(X_{s}^{n})ds+ M_{t}^{n},\quad t\geq0;
\end{eqnarray*}
where for each $n\geq1$, $M^{n}$ is a locally square-integrable martingale such that 
\begin{eqnarray*}
\langle M^{n}{\rangle}_{t}= \int_{0}^{t}{\psi}_{n}(X_{s}^{n})ds, \quad t\geq0;
\end{eqnarray*}
${\varphi}_{n}$ and ${\psi}_{n}$ are Borel measurable functions from $\mathbb{R}$ and ${\mathbb{R}}_{+}$ respectively. We define $V_{t}^{n}= X_{0}^{n}+ \int_{0}^{t}{\varphi}_{n}(X_{s}^{n})ds$. Since our martingales ${\{M_{t}^{n}, t\geq0 \}}$ will be discontinuous, we need to consider their trajectories as elements of  $D([0,\infty))$,  the space of functions from $[0,\infty)$ into $\mathbb{R}$ which are right continuous and have left limits at any $t>0$ (as usual such a function is called càdlàg).We briefly write $\mathbb{D}$ for the space of adapted, càdlàg stochastic processes. We shall always equip the space $D([0,\infty))$with the Skorohod topology, for the definition of which we refer the reader to Billingsley \cite{Billingsley:2009rz} or Joffe, Métivier \cite{Joffe:1986kx}. The following statement can be deduced from Theorem 13.4 and 16.10 of  \cite{Billingsley:2009rz}.
\begin{proposition}\label{ADEU}
A sufficient condition for the above sequence ${\{X_{t}^{n}, t\geq0 \}}_{n\geq1}$ of semi-martingales to be tight in $D([0,\infty))$ is that both
\begin{equation*}
the{~}sequence{~}of{~}r.v.'s {~}\{X_{0}^{n}, n\geq1 \} {~}is {~}tight;
\end{equation*}
and for some $p>1$,
\begin{equation*}
\forall T>0,{~} the {~}sequence {~}of {~}r.v.'s{~} \Bigg\{\int_{0}^{T}{[|{\varphi}_{n}(X_{s}^{n})|+{\psi}_{n}(X_{t}^{n})]}^{p}dt, n\geq1 \Bigg\} {~}is {~}tight.
\end{equation*}
Those conditions imply that both the bounded variation parts $\{V^{n}, n\geq1\}$ and the martingale parts $\{M^{n}, n\geq1\}$ are tight, and that the limit of any converging subsequence of $\{V^{n}\}$ is a.s. continuous.\\
{~~~~}If moreover, for any $T>0$, as $n\longrightarrow \infty$,
\begin{eqnarray*}
\sup_{0\leq t \leq T} |M_{t}^{n}-M_{t^{-}}^{n}|\longrightarrow 0 {~}in {~}probability,
\end{eqnarray*}
then any limit X of a converging subsequence of the original sequence ${\{X^{n}\}}_{n\geq1}$ is a.s. continuous.
\end{proposition}
In particular, the space  $C([0,\infty))$ is closed in  $D([0,\infty))$ equipped with the Skorohod topology. The next Lemma follows from considerations which can be found in  \cite{Billingsley:2009rz}.
\begin{lemma}\label{LDEU}
 Let $X_{n}$, $Y_{n}$ $\in$ $D([0,\infty))$, $n\geq1$ and $X$, $Y$ $\in$ $C([0,\infty))$ be such that \\1. for all $n\geq1$, the function $t\rightarrow Y_{n}(t)$ is increasing;\\2. $X_{n}\rightarrow X$ and $Y_{n}\rightarrow Y$, both locally uniformly.\\ Then $Y$ is increasing and
\begin{eqnarray*}
\int_{0}^{t}X_{n}(s)dY_{n}(s)\rightarrow \int_{0}^{t}X(s)dY(s),{~} locally,{~}uniformly,{~}in,{~} t\geq0.
\end{eqnarray*}
\end{lemma}
The following is a consequent of Theorem 13.5 of  \cite{Billingsley:2009rz}.
\begin{proposition}\label{AUN}
If ${\{X_{t}^{n}, t\geq0 \}}_{n\geq1}$ and ${\{Y_{t}^{n}, t\geq0 \}}_{n\geq1}$ are two tight sequences of random elements of $D([0,\infty))$ and $C([0,\infty))$  respectively, then ${\{X_{t}^{n}+Y_{t}^{n}, t\geq0 \}}_{n\geq1}$
is tight in $D([0,\infty))$.
\end{proposition}
To $x$ $\in$  $D([0,\infty); \mathbb{R})$, we associate for each $T>0$ and $\rho>0$ the quantity
\begin{equation*}
\bar{w}_{x,T}(\rho)= \bar{w}_{T}(x,\rho)= \inf_{\pi \in \Pi_{\rho}^{T}} \max_{t_i \in \pi}\sup_{{ t_i\leq s< t\leq t_{i+1}}} |x(t)-x(s)|, 
\end{equation*}
where $\Pi_{\rho}^{T}$ is the set of all increasing sequences $0=t_0<t_1< \cdot \cdot \cdot <t_n= T $   with the property that $\inf_{0\leq i< n}| t_{i+1}-t_i|\geq \rho$. We state another tightness criterion, which is theorem 13.2 from  \cite{Billingsley:2009rz}.
\begin{proposition}\label{ATROI}
The sequence ${\{X^{n}, n\geq1 \}}$ is tight in $D([0,\infty); \mathbb{R})$ iff the two following conditions hold\\
$(i)$ for each $t\geq0$, $\{X_{t}^{n}-X_{t^{-}}^{n}, n\geq1\}$ is tight in $\mathbb{R}$;

$(ii)$ for each $\epsilon \geq0$, $$\lim_{\rho \rightarrow0} \limsup_{n\rightarrow \infty} \mathbb{P}(\bar{w}_{T}(X^{n},\rho)\geq \epsilon)=0.$$
\end{proposition}
One can compare $\bar{w}_{x,T}(\rho)$ with $w_{x,T}(\rho)$. Consider the maximum (absolute) jump in $x$ :
\begin{equation}\label{JX}
j(x)= \sup_{0< t \leq T} | x(t) - x(t-)|; 
\end{equation} 
the supremum is achieved because only finitely many jumps can exeed a given positive number. \\We have 
\begin{equation}\label{JJXX}
w_{x,T}(\rho) \leq  2\bar{w}_{x,T}(\rho) + j(x)
\end{equation} 
and 
\begin{equation*}
\bar{w}_{x,T}(\rho) \leq w_{x,T}(2\rho)  
\end{equation*} 
( see section 12, page 123 in  \cite{Billingsley:2009rz}).

\par As is well know that, $C$-tightness implies $D$-tightness, in the sense that
\begin{corollary}\label{Aplus}
A sufficient condition for the sequence ${\{X^{n}, n\geq1 \}}$ to be tight in $D([0,\infty); \mathbb{R})$ is that $\forall$ $T\geq1$, $\epsilon \geq0$,
\begin{equation*}
\lim_{\rho \rightarrow0} \limsup_{n\rightarrow \infty} \mathbb{P}({w}_{T}(X^{n},\rho)\geq \epsilon)=0
\end{equation*}
and under that condition any limit of a converging subsequence is continuous (Corollary of Theorem 13.4  in  \cite{Billingsley:2009rz}).  
\end{corollary}
\vskip 1cm
\subsection{Tightness and Weak convergence of \texorpdfstring{$X^{N,x}$}{Lg}} 
The following result describes the limit of the sequence of processes  $\{X^{N,x}, N \geq 1\}$ defined in \eqref{XN}. The continuous time Galton-Watson process $\{X_{t}^{N,x}, t\geq 0\}$  is a Markov process with values in the set $E_{N}=\{{k}/{N}, k\geq 1\}$ with its infinitesimal generator given by:
\begin{eqnarray*}
 Q^{N}f(x)=Nx\left(\frac{N\sigma^2}{2a}+\frac{\alpha}{a}\right)\left[\sum_{\ell\ge 1} p_{\ell}f(x+\frac{\ell}{N})-f(x)\right]+Nx\left(\frac{N\sigma^2}{2}+\beta\right)\left[f(x-\frac{1}{N})-f(x)\right],
\end{eqnarray*}
for any $f:E_{N}\longrightarrow{\mathbb{R}},  \; x \in$ $E_{N}$, where $p_{\ell}$ is the probability that there are $\ell$ simultaneous births. Consequently for any  $f \in  \mathcal{C}(\mathbb{R})$,
\begin{equation*}
M_{t}^{f,N}:=f(X_{t}^{N,x})-f(X_{0}^{N,x})-\int_{0}^{t}Q^{N}f(X_{s}^{N,x})ds
\end{equation*}
is a local martingale. Applying successively the above formula to the cases $f(x)=x$ and $f(x)=x^{2}$, we get that
\begin{equation}\label{CBP}
X_{t}^{N,x}= X_{0}^{N,x}+ (\alpha- \beta)\int_{0}^{t}X_{s}^{N,x}ds+M_{t}^{(1),N}
\end{equation}
and
\begin{align}\label{CBCARR}
{(X_{t}^{N,x})}^{2}&={(X_{0}^{N,x})}^{2}+2(\alpha-\beta)\int_{0}^{t}{(X_{s}^{N,x})}^{2}ds+\left(\frac{\sigma^2N+2\alpha}{2aN}\sum_{\ell\ge 1} \ell^{2}p_{\ell}+\frac{\sigma^2N+2\beta}{2N}\right)\int_{0}^{t}X_{s}^{N,x}ds
\nonumber\\&+M_{t}^{(2),N},
\end{align}
where $\{M_{t}^{(1),N}, t\geq 0\}$ and $\{M_{t}^{(2),N}, t\geq 0\}$ are local martingales. Now combining \eqref{CBP}, \eqref{CBCARR} and the Itô formula, we deduce that
\begin{equation}\label{CRO}
\langle{M}^{(1),N}{\rangle}_{t}=\left(\kappa^2   +\frac{\frac{\alpha}{a}(\zeta^2+a^{2})+\beta}{N}\right)\int_{0}^{t}X_{s}^{N,x}ds.
\end{equation}
where $\kappa^2=\frac{\sigma^2}{2a}(a+a^2+\zeta^2)$. \\ We will establish some lemmas to prove the tightness of the process $X^{N}$.
\begin{lemma}\label{ESXN}
For all $T>0$, there exist a constant $C>0$ such that for all $N\geq1$,
\begin{equation*}
\sup_{0\leq t\leq T}\mathbb{E}({X}_{t}^{N,x})\leq C.
\end{equation*}
\end{lemma}
\begin{proof}
Let $({\tau}_{n}, n\geq0)$ be a sequence of stopping times such that ${\tau}_{n}$ tends to infinity as $n$ goes to infinity and for any $n$, $({M}_{t\wedge {\tau}_{n}}^{(1),N})$ is a martingale and ${X}_{t\wedge {\tau}_{n}}^{N,x}\leq n$. Taking the expectation on both sides of equation \eqref{CBP} at times $t\wedge {\tau}_{n}$ , we obtain 
\begin{eqnarray*}
\mathbb{E}({X}_{t\wedge {\tau}_{n}}^{N,x})= \mathbb{E}({X}_{0}^{N,x})+ (\alpha- \beta)\mathbb{E}(\int_{0}^{t\wedge {\tau}_{n}}X_{s}^{N,x}ds),
\end{eqnarray*}
it follows that
\begin{eqnarray*}
 \mathbb{E}({X}_{t\wedge {\tau}_{n}}^{N,x})\leq \mathbb{E}({X}_{0}^{N,x})+ (\alpha- \beta)^+ \int_{0}^{t}\mathbb{E}(X_{s\wedge {\tau}_{n}}^{N,x})ds.
\end{eqnarray*}
From Gronwall and Fatou Lemmas, we deduce that for all $T>0$ there exists a contant $C>0$ such that
\begin{eqnarray*}
\sup_{N\geq1}\sup_{0\leq t\leq T}\mathbb{E}({X}_{t}^{N,x})\leq C.
\end{eqnarray*}
\end{proof}
\\We shall also need below the
\begin{lemma}\label{ESXCARR}
 For any $T>0$, 
\begin{equation*}
\sup_{N\geq1}\mathbb{E}\left[\sup_{0\leq t\leq T}{({X}_{t}^{N,x})}^{2}\right]< \infty.
\end{equation*}
\end{lemma}
\begin{proof}
Since from \eqref{CBP}, we have
\begin{equation*}
\mathbb{E}\left[\sup_{0\leq t\leq T}{({X}_{t}^{N,x})}^{2}\right]\le 3 \mathbb{E}\left({\left|{X}_{0}^{N,x}\right|}^{2}\right) + 3 (\alpha- \beta)^2 T \int_{0}^{T} \mathbb{E}\left[\sup_{0\leq s\leq t}{({X}_{s}^{N,x})}^{2}\right]ds+ 3\mathbb{E}\left[\sup_{0\leq t\leq T}{\left({M}_{t}^{(1),N}\right)}^{2}\right].
\end{equation*}
Using Doob's inequality, we obtain 
\begin{equation*}
\mathbb{E}\left[\sup_{0\leq t\leq T}{({X}_{t}^{N,x})}^{2}\right]\le 3 \mathbb{E}\left({\left|{X}_{0}^{N,x}\right|}^{2}\right) + 3 (\alpha- \beta)^2 T \int_{0}^{T} \mathbb{E}\left[\sup_{0\leq s\leq t}{({X}_{s}^{N,x})}^{2}\right]ds+ 3 C\mathbb{E}\left( \langle{M}^{(1),N}{\rangle}_{T} \right).
\end{equation*}
However, from  \eqref{CRO}  and  Lemma \ref{ESXN}, we have that
\begin{equation*}
\mathbb{E}(\langle{M}^{(1),N}{\rangle}_{T})\leq \left(\kappa^2 +\frac{\frac{\alpha}{a}(\sigma^2+a^{2})+\beta}{N}\right) CT= CT
\end{equation*}
for all $T>0$. The above computations, combined with Gronwall's Lemma, lead to 
\begin{equation*}
\sup_{N\geq1}\mathbb{E}\left[\sup_{0\leq t\leq T}{({X}_{t}^{N,x})}^{2}\right]< \infty.
\end{equation*}
\end{proof}
\\Recall that, $C$ denotes a constant which may differ from one line to the next.
\begin{corollary}
$\{M_{t}^{(1),N}, t\geq 0\}$ and $\{M_{t}^{(2),N}, t\geq 0\}$ are in fact martingales.
\end{corollary}
It now follows from Proposition \ref{ADEU},  \eqref{CBP}, \eqref{CRO}, Lemma \ref{ESXCARR}  and the fact $X_{0}^{N,x}\longrightarrow x$ that ${\{X^{N,x}\}}_{N\geq1}$ is tight in $D([0,\infty))$.\\ Standard arguments exploiting  \eqref{CBP} and \eqref{CBCARR} now allow us to deduce the convergence of the mass processes (for a detailed proof, see, e.g. Theorem 5.3 p. 23 in \cite{meleard2012quasi}).
 \begin{proposition} 
 We  have  $X^{N,x}\Rightarrow X^x$ as $N \rightarrow \infty $ for the topology of locally uniform convergence, where $X$ is the unique solution of the following Feller SDE : 
 \begin{equation*}
X_{t}^x= x + (\alpha -\beta)\int_{0}^{t} X_{s}^xds + \kappa \int_{0}^{t}\sqrt{X_{s}^x}dW_{s}
\end{equation*}
where $W$ is a standard Brownian motion.
 \end{proposition} 
\subsection{Tightness and Weak convergence of  \texorpdfstring{$H^{N,\Gamma}$}{Lg} }
\subsubsection{Some preliminary results on Galton-Watson branching process}
In this section, we state some results on Galton-Watson branching process which will be useful in checking tightness of $H^{N,\Gamma}$. To do this, we denote by $\Im^N$ the set of finite rooted non binary trees which are defined as follows. An ancestor is born at time $0$. Until she eventually dies, she produces a random number of offsprings. The same happens to each of her offsprings, the offsprings of her offsprings, etc., until eventually the population dies out as all individuals alive at time $ \Gamma^{-}$ are killed at time $ \Gamma$.  In the same way as done in section 2, we denote for any $k>0$, ${\Theta}_{k}$  the number of brothers and sisters born at the $k$th time of birth. We have described in section 2 a bijection between non binary trees and exploration processes. Therefore, we associate to $H^{N,\Gamma}$, the Galton-Watson tree , killed at time $t=\Gamma$ (which is a random element of $\Im^N$) with the same law $\pi$ and the same pair of parameters $(\mu_N,\lambda_N)$ as follows. The lifetime of each individual is exponential with parameter $\mu_{N}$. The birth events arrive according to a Poisson process with rate $\lambda_{N}$ and at each time of birth,  there is a random number of offsprings with law $\pi$. We define $U_k^N$ to be the lifetime of individual $k$. The behaviors of the various individuals are i.i.d. \\Consider the branching process in continuous time $(Z_{t}^{N})_{t\geq 0}$   which describes the number of offspring alive at time $t$ of a unique ancestor, in the population with birth rate $\lambda_N$ and death rate $\mu_N$, whose progeny is killed at time $t=\Gamma$. 
 We define the length of the genealogical tree 
\begin{eqnarray*}
S_{N}^{\Gamma}= \int_{0}^{\Gamma} Z_{t}^{N} dt.
\end{eqnarray*} 
Since the births along this tree are occurring at rate $\lambda_N$, then the total number of offsprings born before time $\Gamma$, denoted $\Lambda_{N}^{\Gamma}$, satisfies
\begin{equation}\label{Whole}
\E\left(\Lambda_{N}^{\Gamma}\Big| S_{N}^{\Gamma}\right)=\lambda_NS_{N}^{\Gamma}.
\end{equation}
We now prove
\begin{lemma}\label{ESTAYTO} 
For any $\Gamma>0$, there exists a constant $C(\Gamma)$ such that
\begin{eqnarray*}
\mathbb{E} (\Lambda_{N}^{\Gamma}) \leqslant   C(\Gamma) N.
\end{eqnarray*} 
\end{lemma}
\begin{proof}
We have from \eqref{Whole}
\begin{equation}\label{BORTAY}
\mathbb{E} (\Lambda_{N}^{\Gamma})= \E(\lambda_NS_{N}^{\Gamma})= \lambda_N \E\left( \int_{0}^{\Gamma} Z_{t}^{N} dt\right)
\end{equation}
However, from \eqref{XN} and \eqref{CBP}, we have that 
\begin{eqnarray*}
\E(Z_{t}^{N})= 1 + (\alpha - \beta) \int_{0}^{t} \E(Z_{s}^{N}) ds, 
\end{eqnarray*}
it is easy to see that  
\begin{eqnarray*}
\E(Z_{t}^{N})=\exp \left[ (\alpha - \beta)t\right] .
\end{eqnarray*}
Hence, 
\begin{equation}\label{ESZNT}
\E\left( \int_{0}^{\Gamma} Z_{t}^{N} dt\right)= (\alpha - \beta)^{-1} \big(\exp\left[ (\alpha - \beta)\Gamma\right]-1\big)= C(\Gamma)
\end{equation}
Now combining \eqref{BORTAY} and \eqref{ESZNT}, we deduce that 
\begin{eqnarray*}
\mathbb{E} (\Lambda_{N}^{\Gamma}) \leqslant   C(\Gamma) N,
\end{eqnarray*} 
since $\lambda_N= {N\sigma^2}/{2a}+{\alpha}/{a}$.
\end{proof}

We  consider a sequence  $({\Delta}_{k}^{N})_{k\geq 1}$ of  i.i.d  random variables which  are   independent   of  $\{{\Theta}_{k}, k \ge 1\}$,  and   such  that 
\begin{equation}\label{DELTALOI}
 {\Delta}_{1}^{N}  \stackrel{(d)}{=}  \sum_{k=1}^{\Lambda_N^{\Gamma}}\frac{{U}_{k}^N}{N}, 
\end{equation}
where $U_{k}^N \sim \mathcal{E} (\mu_{N})$. Recall that $\{U_k^N, \ k\geq1\}$ is a sequence of i.i.d random variables describing the lifetime of individual $k$ and is independent of $\Lambda_N^{\Gamma}$. The  random  variables  ${\Delta}_{k}^{N}$ describes  the   exploration  time  of  the total progeny of one individual. 
\\Let $({T}_{k}^N)_{k\geq 1}$ be the sequence of  i.i.d random variables defined by
\begin{equation}\label{TKEGAL}
{T}_{k}^N = \sum_{i=1}^{{\Theta}_{k}}{\Delta}_{k,i}^{N},
\end{equation} 
describing the time it takes to explore the total progeny of ${\Theta}_{k}$ individuals born at the same time. Note that $\{{\Delta}_{k,i}^N,{~} k\geq 1, i\geq 1\}$ are i.i.d random variables with as law that of ${\Delta}_{1}^N$, and are independent of ${\Theta}_{k}$. 

By combining \eqref{DELTALOI}, \eqref{TKEGAL} and Lemma  \ref{ESTAYTO}, we deduce
\begin{lemma}\label{ESPTUN} 
For any $\Gamma>0$, there exists a constant $C(\Gamma)$ such that
\begin{eqnarray*}
\mathbb{E} ({T}_{1}^N) \leqslant   C(\Gamma) .
\end{eqnarray*} 
\end{lemma}
\begin{proof}
Using Wald's identity, we obtain 
\begin{eqnarray*}
\mathbb{E} ({T}_{1}^N)= \mathbb{E} ({\Theta}_{1}) \mathbb{E}( {\Delta}_{1}^{N} )= \frac{1}{N}\mathbb{E} ({\Theta}_{1}) \mathbb{E}({U}_{1}^N) \mathbb{E}(\Lambda_N^{\Gamma}).
\end{eqnarray*} 
The result now follows from Lemma \ref{ESTAYTO} . 
\end{proof}

Thanks  to  these   results,  we  are  in  position  to  study the  asymptotic  property  of $H^{N,\Gamma}$.
\subsubsection{Tightness and Weak convergence of  \texorpdfstring{$H^{N,\Gamma}$}{Lg}} 
In this section, we will need to write precisely the evolution of  $\{H_{s}^{N,\Gamma}, \ s\geq0\}$, the height process of the forest of trees representing the population $\{Z_{t}^{N,x}, \ 0\leq t\leq \Gamma\}$. To this end, let $\{V_{s}^N, \ s\geq 0\}$  be the càdlàg $\{-1,1\}$-valued process which  is such that, $s-$almost everywhere,  ${dH_{s}^{N,\Gamma}}/{ds}=2NV_{s}^N$ . \\The $({\mathbb{R}}_{+} \times \{-1,1\})$-valued process  $\{(H_{s}^{N,\Gamma}, V_{s}^N), \ s\geq 0\}$ solves the SDE
\begin{align}\label{PREDEFHN}
H_{s}^{N,\Gamma}&=2N\int_{0}^{s}V_{r}^{N}dr,  \quad H_{0}^{N,\Gamma}=0, \quad V_{0}^{N}=1, \nonumber\\
V_{s}^{N}&=1+ 2\int_{0}^{s}\mathbf{1}_{\{V_{r^{-}}^{N}=-1\}}dP_{r}^{N,+} -2\int_{0}^{s}\mathbf{1}_{\{V_{r^{-}}^{N}=+1\}}dP_{r}^{N,-}+ \frac{\kappa^2\delta N}{2} \Big(L_{s}^{N,\Gamma}(0)- L_{0^{+}}^{N,\Gamma}(0)\Big) \nonumber\\ 
&- \frac{\kappa^2\delta N}{2} L_{s}^{N,\Gamma}(\Gamma^{-}) + 2N \sum_{k>0, S_{k}^{N,+}\leq s} \bigg( \frac{\kappa^2\delta}{4} \Big(L_{s}^{N,\Gamma}(H_{S_{k}^{N,+}}^{N,\Gamma})-L_{S_{k}^{N,+}}^{N,\Gamma}(H_{S_{k}^{N,+}}^{N,\Gamma})\Big)\bigg)\wedge \frac{ ({\Theta}_{k}-1)}{N}
\end{align}
where $\{{P}_{s}^{N,+}, s\geq 0\}$ and $\{P_{s}^{N,-}, s\geq 0\}$ are two  mutually independent Poisson processes,  with respective intensities  
$$a^{-1}\delta\left(N^2\kappa^2 +2N\alpha \right)\quad \mbox{and}\quad \delta \left(N^2\kappa^2 + 2N\beta \right)$$
and where the $S_{k}^{N,+}$ are the successive jump times of the process 
\begin{equation}\label{PTILD}
\widetilde{P}_{s}^{N,+}= \int_{0}^{s}\mathbf{1}_{\{V_{r^{-}}^{N}=-1\}}dP_{r}^{N,+}
\end{equation}
and where  $L_{s}^{N,\Gamma}(0)$ and $L_{s}^{N,\Gamma}(\Gamma^{-})$  respectively denote the number of visits to $0$ and $\Gamma$ by the process $H^{N,\Gamma}$ up to time $s$, multiplied by $4 / N\kappa^2\delta$  (see \eqref{TL}). Note that our definition of $L^{N,\Gamma}$ makes the mapping $t$ $\rightarrow$ $L_{s}^{N,\Gamma}(t)$ right continuous for each $s>0$. Hence $L_{s}^{N,\Gamma}(t)=0$ for $t\geq\Gamma$, while  $L_{s}^{N,\Gamma}(\Gamma^{-})= \lim_{n\rightarrow +\infty} \ L_{s}^{N,\Gamma}(\Gamma- \frac{1}{n})>0$ if $H^{N,\Gamma}$ has reached the level $\Gamma$ by time $s$.\\
For the rest of this section we set $$ \nu= \kappa \sqrt{ \delta} .$$ We deduce from \eqref {PREDEFHN}
\begin{equation*}
\frac{V_{s}^{N}}{N\nu^2}= \frac{1}{N\nu^2}+ Q_{s}^{N,+} - \frac{2}{N \nu^2}\int_{0}^{s}\mathbf{1}_{\{V_{r^{-}}^{N}=+1\}}dP_{r}^{N,-}+ \frac{1}{2} \Big(L_{s}^{N,\Gamma}(0)- L_{0^{+}}^{N,\Gamma}(0)\Big)- \frac{1}{2} L_{s}^{N,\Gamma}(\Gamma^{-}) ,
\end{equation*}
where
\begin{align}\label{QUNQDEU}
Q_{s}^{N,+}&=  \frac{2}{N \nu^2} \int_{0}^{s}\mathbf{1}_{\{V_{r^{-}}^{N}=-1\}}dP_{r}^{N,+} + \frac{2}{\nu^2} \sum_{k>0, S_{k}^{N,+}\leq s} \bigg( \frac{\nu^2}{4}(L_{s}^{N,\Gamma}(H_{S_{k}^{N,+}}^{N,\Gamma})-L_{S_{k}^{N,+}}^{N,\Gamma} (H_{S_{k}^{N,+}}^{N,\Gamma}))\bigg)\wedge \frac{ ({\Theta}_{k}-1)}{N}\nonumber\\
& =   \frac{2}{N\nu^2} \int_{0}^{s} \left[  1+  \frac{N\nu^2}{4}(L_{s}^{N,\Gamma}(H_{r}^{N,\Gamma})-L_{r}^{N,\Gamma}(H_{r}^{N,\Gamma}))\wedge ({\Theta}_{\widetilde{P}_{r}^{N,+}}-1)  \right]d\widetilde{P}_{r}^{N,+}\nonumber\\
 & =   Q_{s}^{N,+,1} - Q_{s}^{N,+,2}, 
\end{align}
with
\begin{align*}
 \quad Q_{s}^{N,+,1}=  \frac{2}{N\nu^2} \int_{0}^{s} {{\Theta}_{\widetilde{P}_{r^{-}}^{N,+}+1}}d\widetilde{P}_{r}^{N,+} \quad \mbox{ and} 
\end{align*}
\begin{equation}\label{QDEPLUS}
Q_{s}^{N,+,2} =\frac{2}{N\nu^2} \int_{0}^{s} \left({{\Theta}_{\widetilde{P}_{r^{-}}^{N,+}+1}}  -1- \frac{N\nu^2}{4}(L_{s}^{N,\Gamma}(H_{r}^{N,\Gamma})-L_{r}^{N,\Gamma}(H_{r}^{N,\Gamma})) \right)^{+} d\widetilde{P}_{r}^{N,+}.
\end{equation}
Writing the first line of \eqref {PREDEFHN} as
\begin{equation*}
 H_{s}^{N,\Gamma}= 2N\int_{0}^{s}\mathbf{1}_{\{V_{r}^{N}=+1\}}dr - 2N\int_{0}^{s}\mathbf{1}_{\{V_{r}^{N}=-1\}}dr,
\end{equation*}
denoting by $M_{s}^{1,N}$ and $M_{s}^{2,N}$ the two martingales
\begin{equation}\label{PREMUN}
M_{s}^{1,N} =  Q_{s}^{N,+,1} - \frac{2}{N\nu^2}   \int_{0}^{s}\mathbf{1}_{\{V_{r}^{N}=-1\}}{{\Theta}_{\widetilde{P}_{r}^{N,+}+1}} \Big(\frac{N^2\nu^2+2N\alpha \delta}{a}\Big)dr 
\end{equation}
and
\begin{equation}\label{PREMDEU}
M_{s}^{2,N} = \frac{2}{N\nu^2}   \int_{0}^{s}\mathbf{1}_{\{V_{r^{-}}^{N}=+1\}}\left(dP_{r}^{N,-}-(\nu^2N^{2}+2N\beta\delta)dr\right)\quad\quad
\end{equation}
and recalling  \eqref{QUNQDEU}, we deduce from  \eqref{PREDEFHN} ,
\begin{align}\label{DEUDFHN}
 H_{s}^{N,\Gamma}+\frac{V_{s}^{N}}{N\nu^2}&= \frac{1}{N\nu^2}+ M_{s}^{1,N}-M_{s}^{2,N} +\frac{2N}{a}\int_{0}^{s}\mathbf{1}_{\{V_{r}^{N}=-1\}} \big({{\Theta}_{\widetilde{P}_{r}^{N,+}+1}}- a\big)dr- Q_{s}^{N,+,2}  - \frac{1}{2} L_{s}^{N,\Gamma}(\Gamma^{-}) \nonumber\\
&+ \frac{4\alpha \delta}{a\nu^2}\int_{0}^{s}\mathbf{1}_{\{V_{r}^{N}=-1\}}{{\Theta}_{\widetilde{P}_{r}^{N,+}+1}}dr- \frac{4\beta \delta}{\nu^2} \int_{0}^{s}\mathbf{1}_{\{V_{r}^{N}=+1\}}dr+ \frac{1}{2}(L_{s}^{N,\Gamma}(0)- L_{0^{+}}^{N,\Gamma}(0)).
\end{align}
We first check that 
\begin{lemma}\label{QEGAK}
For any $s>0$
\begin{equation*}
\E(Q_{s}^{N,+,2})= \E (K_{s}^{N}),
\end{equation*}
where
\begin{eqnarray*}
K_s^N=  \left(\frac{2N\nu^2+4\alpha \delta}{a\nu^2}\right) \int_{0}^{s}\mathbf{1}_{\{V_{r}^{N}=-1\}} \left( {\Theta}_{\widetilde{P}_{r}^{N,+}+1}-1- \frac{N\nu^2}{4}(L_{s}^{N,\Gamma}(H_{r}^{N,\Gamma})-L_{r}^{N,\Gamma}(H_{r}^{N,\Gamma})) \right)^{+}dr.
\end{eqnarray*}
\end{lemma}
\begin{proof}
 We have
\begin{align*}
\E (Q_{s}^{N,+,2}) &= \dfrac{2}{N\nu^2}  \E \left( \int_{0}^{s} \Big({{\Theta}_{\widetilde{P}_{r}^{N,+}}}  -1- \dfrac{N\nu^2}{4}(L_{s}^{N,\Gamma}(H_{r}^{N,\Gamma})-L_{r}^{N,\Gamma}(H_{r}^{N,\Gamma}))\Big)^{+} d\widetilde{P}_{r}^{N,+}\right) \\
&= \dfrac{2}{N\nu^2}  \E  \left(\sum_{k>0, S_{k}^{N,+}\leq s} \Big({\Theta}_{k}-1- \dfrac{N\nu^2}{4}(L_{s}^{N,\Gamma}(H_{S_{k}^{N,+}}^{N,\Gamma})-L_{S_{k}^{N,+}}^{N,\Gamma}(H_{S_{k}^{N,+}}^{N,\Gamma}))\Big)^{+}\right) \\
&=  \dfrac{2}{N\nu^2}  \sum_{k>0}  \E \left( \mathbf{1}_{\{S_{k}^{N,+}\leq s\}} \Big({\Theta}_{k}-1- \dfrac{N\nu^2}{4}(L_{s}^{N,\Gamma}(H_{S_{k}^{N,+}}^{N,\Gamma})-L_{S_{k}^{N,+}}^{N}(H_{S_{k}^{N,+}}^{N,\Gamma}))\Big)^{+}\right)\\
&=  \dfrac{2}{N\nu^2}  \sum_{k>0}  \E \left(\Phi_{s} (S_{k}^{N,+}, {\Theta}_{k}) \right),
\end{align*}
where 
\begin{align*}
\Phi_{s} (r, n)&=  \mathbf{1}_{\{r\leq s\}} \E \left( \Big(n -1- \dfrac{N\nu^2}{4}(L_{s}^{N,\Gamma}(H_{r}^{N,\Gamma})-L_{r}^{N,\Gamma}(H_{r}^{N,\Gamma}))\Big)^{+} \Big| S_{k}^{N,+}=r, {\Theta}_{k}=n\right)\\
&=  \mathbf{1}_{\{r\leq s\}}  \E \left( \Big(n -1- \dfrac{N\nu^2}{4}(L_{s}^{N,\Gamma}(H_{r}^{N,\Gamma})-L_{r}^{N,\Gamma}(H_{r}^{N,\Gamma}))\Big)^{+} \right).
\end{align*}
We deduce that 
\begin{align*}
\E (Q_{s}^{N,+,2}) &= \dfrac{2}{N\nu^2}   \int_{0}^{\infty} \E \left( \Phi_{s} (r, {\Theta}_{\widetilde{P}_{r}^{N,+}}) \right)d\widetilde{P}_{r}^{N,+}\\
&=   \dfrac{2}{N\nu^2}  \E  \int_{0}^{\infty}  \Phi_{s} (r, {\Theta}_{\widetilde{P}_{r^{-}}^{N,+}+1}) d\widetilde{P}_{r}^{N,+}\\
&=   \left(\frac{2N\nu^2+4\alpha \delta}{a\nu^2}\right)   \E  \int_{0}^{\infty}  \Phi_{s} (r, {\Theta}_{\widetilde{P}_{r}^{N,+}+1}) \mathbf{1}_{\{V_{r}^{N}=-1\}}dr\\
&=   \left(\frac{2N\nu^2+4\alpha \delta}{a\nu^2}\right) \int_{0}^{s} \E \bigg(  \mathbf{1}_{\{V_{r}^{N}=-1\}} \E  \Big({\Theta}_{\widetilde{P}_{r}^{N,+}+1} -1- \dfrac{N\nu^2}{4}(L_{s}^{N,\Gamma}(H_{r}^{N,\Gamma})-L_{r}^{N,\Gamma}(H_{r}^{N,\Gamma}))\Big)^{+}\bigg)dr \\
&=   \left(\frac{2N\nu^2+4\alpha \delta}{a\nu^2}\right) \E \int_{0}^{s} \bigg(  \mathbf{1}_{\{V_{r}^{N}=-1\}} \Big({\Theta}_{\widetilde{P}_{r}^{N,+}+1} -1- \dfrac{N\nu^2}{4}(L_{s}^{N,\Gamma}(H_{r}^{N,\Gamma})-L_{r}^{N,\Gamma}(H_{r}^{N,\Gamma}))\Big)^{+}\bigg)dr\\
&= \E (K_{s}^{N}).
\end{align*}
\end{proof}
\\The next proposition will also be important in the proof of the tightness  and  weak convergence of  $H^{N,\Gamma}$. 
\begin{proposition}\label{ESPOURNS} We  have  
\begin{align*}
\mathbb{E} (K_{s}^{N})\longrightarrow 0, \quad as \quad N\rightarrow \infty.
\end{align*}
\end{proposition}
\begin{proof}
For   $r \in {\mathbb{R}}^{+}$ and   $p \in {\mathbb{N}}^{\ast}$,  the   stopping  time  
 \begin{equation}\label{TAURP}
 \tau_{r}^p= \inf \left\{s\ge 0, \  \frac{N\nu^2}{4}\left(L_{s}^{N,\Gamma}(H_{r}^{N,\Gamma})-L_{r}^{N,\Gamma}(H_{r}^{N,\Gamma})\right) \ge  p \right\}-r
 \end{equation} 
describes the time it takes to explore the offspring of $p$ individuals born at the real time $H_{r}^{N,\Gamma}$. Since, the r.v. ${\Delta}_{k,i}^N$ describes the exploration time of the total progeny of one individual 
, one  can see   that   $\tau_{r}^p = \sum_{i=1}^{p} N^{-2}{\Delta}_{k,i}^N$. Hence   we  deduce   from   \eqref{TKEGAL} and  \eqref{TAURP}  
\begin{align*}
\Big\{  \frac{N\nu^2}{4}(L_{s}^{N,\Gamma}(H_{r}^{N,\Gamma})-L_{r}^{N,\Gamma}(H_{r}^{N,\Gamma}))< {\Theta}_{\widetilde{P}_{r}^{N,+}}-1\Big\}= \Big\{r+ {N^{-2}}{T_{{\widetilde{P}_{r}^{N,+}}}^N} > s \Big \}.
\end{align*} 
Now we have
\begin{align*}
K_{s}^{N}&=\left(\frac{2N\nu^2+4\alpha \delta}{a\nu^2}\right) \int_{0}^{s}\mathbf{1}_{\{V_{r}^{N}=-1\}} \left( {\Theta}_{\widetilde{P}_{r}^{N,+}+1}-1- \frac{N\nu^2}{4}\left(L_{s}^{N,\Gamma}(H_{r}^{N,\Gamma})-L_{r}^{N,\Gamma}(H_{r}^{N,\Gamma})\right) \right)^{+}dr\\
&\le   \left(\frac{2N\nu^2+4\alpha \delta}{a\nu^2}\right)  \int_{0}^{s}\mathbf{1}_{\{V_{r}^{N}=-1\}} {\Theta}_{\widetilde{P}_{r}^{N,+}+1} \mathbf{1}_{\big\{  \frac{N\nu^2}{4}\left(L_{s}^{N,\Gamma}(H_{r}^{N,\Gamma})-L_{r}^{N,\Gamma}(H_{r}^{N,\Gamma})\right)< \ {\Theta}_{\widetilde{P}_{r}^{N,+}+1}-1\big\}}dr\\
&\le  \left(\frac{2N\nu^2+4\alpha \delta}{a\nu^2}\right)  \int_{0}^{s}\mathbf{1}_{\{V_{r}^{N}=-1\}} {\Theta}_{\widetilde{P}_{r}^{N,+}+1} \mathbf{1}_{\big\{ r+ {N^{-2}} {T_{{\widetilde{P}_{r}^{N,+}+1}}^N}> \ s \big\}}dr. 
\end{align*}
Hence   tacking  expectation  in   both side, we  deduce  that  
\begin{align*}
\E(K_{s}^{N}) &\le \left(\frac{2N\nu^2+4\alpha \delta}{a\nu^2}\right)  \mathbb{E} \left( \int_{0}^{s}\mathbf{1}_{\{V_{r}^{N}=-1\}} {\Theta}_{\widetilde{P}_{r}^{N,+}+1} \mathbf{1}_{\big\{ r+ {N^{-2}} {T_{{\widetilde{P}_{r}^{N,+}+1}}^N} > \ s \big\}}dr\right)
\\ &\le      \left(\frac{2N\nu^2+4\alpha \delta}{\nu^2}\right) \int_{0}^{s}  \mathbb{P}({T}_{1}^N> {N^{2}}(s-r)) dr. 
\end{align*}
Let $\frac{1}{2}<\eta <1$. Using Markov's and Jensen's inequality, we obtain
\begin{equation*}
\E(K_{s}^{N}) \le      \left(\frac{2N\nu^2+4\alpha \delta}{\nu^2 N^{2\eta }}\right)\left[\E({T}_{1}^N)\right]^\eta   \int_{0}^{s}   \frac{dr}{(s-r)^\eta } . 
\end{equation*}
From Lemma \ref{ESPTUN} , there exists a constant $C(\Gamma)$ such that
$$\mathbb{E} ({T}_{1}^N) \leqslant   C(\Gamma).$$
This implies 
\begin{equation}\label{Max}
\E(K_{s}^{N}) \le \left(\frac{2N\nu^2+4\alpha \delta}{\nu^2 N^{2^\eta}}\right) \left[C(\Gamma)\right]^\eta   \int_{0}^{s}   \frac{dr}{(s-r)^\eta} . 
\end{equation}
Hence    
\begin{equation*}
\mathbb{E} (K_{s}^{N}) \rightarrow 0, {~} as {~} N\rightarrow \infty.
\end{equation*}
\end{proof}

It follows from Lemma \ref{QEGAK} and Proposition \ref{ESPOURNS} that  for each $s$
\begin{equation*}
Q_{s}^{N,+,2} \rightarrow 0 {~} in{~} probability, {~}as{~} N\rightarrow \infty .
\end{equation*}
Since $s\rightarrow Q_{s}^{N,+,2}$ is increasing, this convergence is locally uniform in $s$; We have the
\begin{corollary}
$Q^{N,+,2} \longrightarrow 0$ in probability in $ \mathcal{C}([0,\infty), {\mathbb{R}}_{+})$.
\end{corollary}
\begin{remark}\label{HNINFGAMA}
From the definition of $H^{N,\Gamma}$ , we have,  $H_{s}^{N,\Gamma}\leq \Gamma$,  $\forall$ $s>0$.
\end{remark}
For the proof of the weak convergence of  $\{H_s^{N,\Gamma}, \ s\ge 0\}$, we will need the following lemma 
\begin{lemma}\label{VTENDMI}
For any $s>0$,
\begin{eqnarray*}
\int_{0}^{s}\mathbf{1}_{\{V_{r}^{N}=1\}}dr\longrightarrow \frac{s}{2}; {~~~}\int_{0}^{s}\mathbf{1}_{\{V_{r}^{N}=-1\}}dr\longrightarrow \frac{s}{2}
\end{eqnarray*}
in probability, as $N\longrightarrow \infty$.
\end{lemma}
\begin{proof}
We have (the second line follows from \eqref{PREDEFHN})
\begin{eqnarray*}
\int_{0}^{s}\mathbf{1}_{\{V_{r}^{N}=1\}}dr+\int_{0}^{s}\mathbf{1}_{\{V_{r}^{N}=-1\}}dr=s,
\end{eqnarray*}
\begin{eqnarray*}
\int_{0}^{s}\mathbf{1}_{\{V_{r}^{N}=1\}}dr-\int_{0}^{s}\mathbf{1}_{\{V_{r}^{N}=-1\}}dr={(2N)}^{-1}H_{s}^{N,\Gamma}.
\end{eqnarray*}
We conclude by adding and substracting the two above identities and using Remark \ref{HNINFGAMA}.
\end{proof}
\\For $s>0$, define $$\widetilde{P}_{s}^{N,-}= \int_{0}^{s}\mathbf{1}_{\{V_{r^{-}}^{N}=-1\}}dP_{r}^{N,-},$$ where $P_{r}^{N,-}$ was defined in \eqref{PREDEFHN}. Let $(\lambda_{+}^N(s), s\geq0)$ $($ resp. $(\lambda_{-}^N(s), s\geq0))$  denote the intensity of the process $({\widetilde{P}_{s}^{N,+}}, s\geq0)$  $($resp.$({\widetilde{P}_{s}^{N,-}}, s\geq0))$ where ${\widetilde{P}^{N,+}}$ was defined in \eqref{PTILD}. More precisely, the process ${\widetilde{P}^{N,+}}$ which counts the successive local minima of $H^{N,\Gamma}$ (except the at height $0$) is a point process with predictable intensity $ \lambda_{+}^N(s)= a^{-1}({N^2\nu^2+2N\alpha \delta})\mathbf{1}_{\{V_{s^{-}}^{N}=-1\}}$, and the process $P_{r}^{N,-}$ which counts the successive local maxima of $H^{N,\Gamma}$  is a point process with predictable intensity $\lambda_{-}^N(s)= ({N^2\nu^2+2N\beta \delta})\mathbf{1}_{\{V_{s^{-}}^{N}=+1\}}$. Recall that the process $V^N$ is the (càdlàg) sign of the slope of $H^{N,\Gamma}$.
\begin{remark}\label{PETPPRIM}
Note that ${\widetilde{P}^{N,+}}$ $($resp. ${\widetilde{P}^{N,-}})$ with intensity $\lambda_{+}^N(.)$ $($resp.  $\lambda_{-}^N(.))$ can be viewed as time-changed of $P$ (resp. $P^{\prime}$), a standard Poisson process i.e. $${\widetilde{P}_{s}^{N,+}}=P\left(\int_{0}^{s}\lambda_{+}^N(r)dr\right)  \quad and \quad {\widetilde{P}_{s}^{N,-}}=P^{\prime}\left(\int_{0}^{s}\lambda_{-}^N(r)dr\right) .$$  
\end{remark}
For the rest of this section we set $$ A_{s}^{N}= \int_{0}^{s}\lambda_{+}^N(r)dr, \quad \hat{A}_{s}^{N}= \int_{0}^{s}\lambda_{-}^N(r)dr, \quad \bar{A}_{s}^{N}={(N^2\nu^2+2N\alpha \delta)}s/{2a}, $$ $$ \Delta_{s}^N= {N^2\nu^2}s/{2a}, \quad \hat{\Delta}_{s}^N= {N^2\nu^2}s/{2}\quad \mbox{and} \quad \tilde{\Theta}_k=  {\Theta}_k-a.$$
In the equation \eqref{DEUDFHN}, we set  
\begin{equation*}
\Psi_{s}^{N}= \frac{2N}{a}\int_{0}^{s}\mathbf{1}_{\{V_{r}^{N}=-1\}} \big({{\Theta}_{\widetilde{P}_{r}^{N,+}+1}}- a\big)dr.
\end{equation*}
From Remark \ref{PETPPRIM}, we deduce that
\begin{align}\label{PRPSIN}
 \Psi_{s}^{N}&=\frac{2}{N\nu^2+2\alpha \delta}\int_{0}^{s} { \tilde{\Theta}_{P(A_{r}^{N})+1}}dA_{r}^{N}= \frac{2}{N\nu^2+2\alpha \delta}\int_{0}^{A_{s}^{N}}{ \tilde{\Theta}_{P(u)+1}}du \nonumber \\ 
&= \frac{2}{N\nu^2+2\alpha \delta} \bigg(\sum_{k=0}^{P(A_{s}^{N})} { \tilde{\Theta}_{k+1}} \Xi_k-{ \tilde{\Theta}_{P(A_{s}^{N})+1}}({T}^+(A_{s}^{N}) - A_{s}^{N})\bigg),
\end{align}
where $ \Xi_k$ denotes the length of the time interval during which $P(u)=k$ and  ${T}^+(A_{s}^{N})$ is the first jump time of $P$ after $A_{s}^{N}$. It is easily seen that $\Xi_k$ has the standard exponential distribution and we notice that  $({\Xi}_1, {\Theta}_1, {\Xi}_2, {\Theta}_2,\cdot \cdot \cdot)$ is a sequence of independent random variable. By the same computations, we deduce from \eqref{PREMUN} that
\begin{equation*}
M_{s}^{1,N}= -\frac{2}{N\nu^2} \bigg(\sum_{k=0}^{P(A_{s}^{N})} { {\Theta}_{k+1}}  \tilde{\Xi}_k- {\Theta_{P(A_{s}^{N})+1}}({T}^+(A_{s}^{N}) - A_{s}^{N})\bigg),
\end{equation*}
where $ \tilde{\Xi}_k= {\Xi}_k- \E({\Xi}_k)$.  From \eqref{PREMDEU}, we have also
\begin{equation*}
M_{s}^{2,N}= -\frac{2}{N\nu^2} \bigg(\sum_{k=0}^{P^{\prime}(\hat{A}_{s}^{N})}\tilde{\Xi}^{\prime}_k- ({T}^-(\hat{A}_{s}^{N}) - \hat{A}_{s}^{N})\bigg),
\end{equation*}
where $ \tilde{\Xi}^{\prime}_k= {\Xi}^{\prime}_k - \E({\Xi}^{\prime}_k)$ and where ${\Xi}^{\prime}_k$ denotes the length of the time interval during which $P^{\prime}(u)=k$ and  ${T}^-(\hat{A}_{s}^{N})$ is the first jump time of $P^{\prime}$ after $\hat{A}_{s}^{N}$. As previously ${\Xi}^{\prime}_k$ has the standard exponential distribution and that   $({\Xi}^{\prime}_1,{\Xi}_1, {\Theta}_1, {\Xi}^{\prime}_2,{\Xi}_2, {\Theta}_2,\cdot \cdot \cdot)$ is a sequence of independent random variable. \\If we define for $n \ge1$
$$S_n^1=\sum_{k=0}^{n}\left(\tilde{\Theta}_{k+1} \Xi_{k}- {\Theta}_{k+1} \tilde{\Xi}_k\right),\quad \bar{S}_n^1=\sum_{k=0}^{n}\tilde{\Theta}_{k+1} \Xi_{k} \quad \mbox{and} \quad S_n^2=\sum_{k=0}^{n}\tilde{\Xi}^{\prime}_k,$$
we obtain the following relations
\begin{equation}\label{PSIPLUMUN}
 \Psi_{s}^{N}+M_{s}^{1,N}= \frac{2}{N\nu^2} S_{P(A_{s}^{N})}^1- \frac{2}{N\nu^2} \bar{S}_{P(A_{s}^{N})}^1 C_N - \Lambda_{A_{s}^{N}}^{1} + \bar{\Lambda}_{A_{s}^{N}}^{1}, 
\end{equation}
 with 
 \begin{equation}\label{LAMDUN}
 \Lambda_{A_{s}^{N}}^{1}= \frac{2}{N\nu^2+2\alpha \delta}{ \tilde{\Theta}_{P(A_{s}^{N})+1}}({T}^+(A_{s}^{N}) - A_{s}^{N}), 
 \end{equation}
$$\bar{\Lambda}_{A_{s}^{N}}^{1}= \frac{2}{N\nu^2}{ {\Theta}_{P(A_{s}^{N})+1}}({T}^+(A_{s}^{N}) - A_{s}^{N}), \quad C_N=\frac{2\alpha \delta}{(N\nu^2+2\alpha \delta)}$$
and
\begin{equation}\label{SECONMDEU}
M_{s}^{2,N}= -\frac{2}{N\nu^2} S_{P^{\prime}(\hat{A}_{s}^{N})}^2 + \Lambda_{\hat{A}_{s}^{N}}^{2} 
\end{equation}
with 
$$ \Lambda_{\hat{A}_{s}^{N}}^{2}= \frac{2}{N\nu^2}({T}^-(\hat{A}_{s}^{N}) - \hat{A}_{s}^{N}).$$
The following Proposition plays a key role in the asymptotic behavior of $H^{N,\Gamma}$ 
\begin{proposition}\label{PROPSIMUN} 
As $N\longrightarrow\infty$, $$\Big( \Psi_{s}^{N}+M_{s}^{1,N},M_{s}^{2,N}, \ s\geq0\Big)\Longrightarrow \left( \frac{\sqrt{2}}{\nu} \sqrt{\frac{a^2+\zeta^2}{a}} B_{s}^{1},  \frac{\sqrt{2}}{\nu} B_{s}^{2}, \ s\geq0\right) \ in \ {(\mathcal{D}([0,\infty)))}^{2},$$ where $B_{s}^{1}$ and $B_{s}^{2}$ are two mutually independent standard Brownian motions.
\end{proposition}
We first prove the 
\begin{lemma}\label{LEMLAMUN}
As $N\longrightarrow\infty$,  $\Lambda_{A_{s}^{N}}^{1}\longrightarrow 0 \ \mbox{ in} \ \mbox{probability} \ \mbox{locally} \ \mbox{uniformly} \ \mbox{in} \ \mbox{s} $.
\end{lemma}
\begin{proof}
Let $T>0$. From \eqref{LAMDUN}, we notice that
\begin{equation*}
\big| \Lambda_{A_{s}^{N}}^{1} \big| \leq \frac{2}{N\nu^2+2\alpha \delta}\big|{ \tilde{\Theta}_{P(A_{s}^{N})+1}}\big|{ \Xi_{P(A_{s}^{N})+1}},
\end{equation*}
this implies
\begin{equation}\label{SUPLAMANS}
\sup_{0\leqslant s \leqslant T} \big| \Lambda_{A_{s}^{N}}^{1} \big| \leq \sup_{0\leq k\leq P(A_{T}^{N})} \left(\frac{2}{N\nu^2+2\alpha \delta}\big| \tilde{\Theta}_{k+1}\big| \Xi_{k+1}\right)
\end{equation}
However, for $\rho>0$, we note that 
\begin{align}\label{TETAINCLU}
\Big\{\sup_{0\leq k\leq P(A_{T}^{N})} \big| \tilde{\Theta}_{k+1}\big| \Xi_{k+1} > \frac{\rho(N\nu^2+2\alpha \delta)}{2} \Big\} &\subset  \Big\{ \sup_{0\leq k\leq 2A_{T}^{N}} \big| \tilde{\Theta}_{k+1}\big| \Xi_{k+1} > \frac{\rho(N\nu^2+2\alpha \delta)}{2} \Big\}\nonumber\\ &\cup \Big\{  P(A_{T}^{N}) > 2 A_{T}^{N} \Big\}. 
\end{align}
It follows from \eqref{SUPLAMANS} and \eqref{TETAINCLU} that 
\begin{align*}
\mathbb{P} \left(\sup_{0\leqslant s \leqslant T} \big| \Lambda_{A_{s}^{N}}^{1} \big|>\rho\right)&\leq \mathbb{P} \left(\sup_{0\leq k\leq 2A_{T}^{N}} \big| \tilde{\Theta}_{k+1}\big| \Xi_{k+1} > \frac{\rho(N\nu^2+2\alpha \delta)}{2} \right)+ \mathbb{P} \left(P(A_{T}^{N}) > 2 A_{T}^{N}  \right)\\&\leq \mathbb{P} \left(\sup_{0\leq k\leq 2\bar{A}_{T}^{N}} \big| \tilde{\Theta}_{k+1}\big| \Xi_{k+1} > \frac{\rho(N\nu^2+2\alpha \delta)}{2} \right)+ \mathbb{P} \left(P(A_{T}^{N}) > 2 A_{T}^{N}  \right)
\end{align*}
We deduce from the law of large numbers that the second term on the right converges to 0 a.e, as $N\longrightarrow\infty$. We will now show that the first term on the right converges to 0, as $N\longrightarrow\infty$. We have
\begin{align*}
&\mathbb{P} \left(\sup_{0\leq k\leq 2\bar{A}_{T}^{N}} \big| \tilde{\Theta}_{k+1}\big| \Xi_{k+1} > \frac{\rho(N\nu^2+2\alpha \delta)}{2} \right)= 1-\left[1- \mathbb{P} \left(\big| \tilde{\Theta}_{1}\big| \Xi_{1} > \frac{\rho(N\nu^2+2\alpha \delta)}{2} \right) \right]^{2\bar{A}_{T}^{N}}\nonumber\\&= 1-\left[1- \mathbb{E} \left(\exp \left(-\frac{\rho(N\nu^2+2\alpha \delta)}{2\big| \tilde{\Theta}_{1}\big| } \right) \right) \right]^{2\bar{A}_{T}^{N}}
\end{align*}
however, we notice that, as $N$ tends to infinity, 
\begin{equation*}
\left[1- \mathbb{E} \left(\exp \left(-\frac{\rho(N\nu^2+2\alpha \delta)}{2\big| \tilde{\Theta}_{1}\big| } \right) \right) \right]^{2\bar{A}_{T}^{N}}\simeq \ \exp\left[-2\bar{A}_{T}^{N}  \mathbb{E} \left(\exp \left(-\frac{\rho(N\nu^2+2\alpha \delta)}{2\big| \tilde{\Theta}_{1}\big| } \right) \right) \right].
\end{equation*}
Let 
\begin{equation*}
b_N=\bar{A}_{T}^{N}  \mathbb{E} \left(\exp \left(-\frac{\rho(N\nu^2+2\alpha \delta)}{2\big| \tilde{\Theta}_{1}\big| } \right) \right) 
\end{equation*}
We now show that  $b_N \longrightarrow 0,  \ as  \ N\longrightarrow\infty,$  which will imply the Lemma. We have 
\begin{align*}
b_N&= \frac{(N^2\nu^2+2N\alpha \delta)T}{2a}\int_{0}^{1} \mathbb{P}  \left(\exp \left(-\frac{\rho(N\nu^2+2\alpha \delta)}{2\big| \tilde{\Theta}_{1}\big| }  \right)> t \right) dt\\&=\frac{\rho NT}{a} \left(\frac{N\nu^2+2\alpha \delta}{2}\right)^2\int_{0}^{\infty}y^{-2} \mathbb{P} (\big| \tilde{\Theta}_{1}\big|>y) \exp \left(-\frac{\rho(N\nu^2+2\alpha \delta)}{2y }  \right)dy
\end{align*}
(recall that $\bar{A}_{T}^{N}=  {(N^2\nu^2+2N\alpha \delta)}T/{2a}$). Define for $N\geq1$, 
\begin{equation*}
f_N(y)= N\left(\frac{N\nu^2+2\alpha \delta}{2}\right)^2 y^{-2} \mathbb{P} (\big| \tilde{\Theta}_{1}\big|>y) \exp \left(-\frac{\rho(N\nu^2+2\alpha \delta)}{2y }  \right).
\end{equation*}
It is easily seen that $f_N(y) \longrightarrow 0,  \ as  \ N\longrightarrow\infty$ and it is not very hard to show that
\begin{equation*}
f_N(y)\leq \frac{6^3}{4\rho^3}e^{-3} y \mathbb{P} (\big| \tilde{\Theta}_{1}\big|>y).
\end{equation*}
Hence, since $\tilde{\Theta}_{1}$ is square integrable, we deduce from the dominated convergence theorem that 
\begin{equation*}
b_N \longrightarrow 0,  \ as  \ N\longrightarrow\infty.
\end{equation*}
\end{proof}
\\Following the same appoach, we have the 
\begin{lemma}\label{LADBARCHAP}
As $N\longrightarrow\infty$,  $\bar{\Lambda}_{A_{s}^{N}}^{1}\longrightarrow 0 \ (resp. \ \Lambda_{\hat{A}_{s}^{N}}^{2}\longrightarrow 0) \ \mbox{ in} \ \mbox{probability} \ \mbox{locally} \ \mbox{uniformly} \ \mbox{in} \ \mbox{s}.$
\end{lemma}
Let us rewrite \eqref{PSIPLUMUN} in the form 
\begin{align}\label{GUNBACHA}
 \Psi_{s}^{N}+M_{s}^{1,N}&= \frac{2}{N\nu^2} S_{[\Delta_{s}^{N}]}^1 +  \frac{2}{N\nu^2} \left(S_{P(A_{s}^{N})}^1-S_{[\Delta_{s}^{N}]}^1\right) -\frac{2}{N\nu^2}\bar{S}_{P(A_{s}^{N})}^1C_N - \Lambda_{A_{s}^{N}}^{1} + \bar{\Lambda}_{A_{s}^{N}}^{1},\nonumber \\
& = G_{s}^{1,N}+ \bar{G}_{s}^{1,N}- \hat{G}_{s}^{1,N}- \Lambda_{A_{s}^{N}}^{1} + \bar{\Lambda}_{A_{s}^{N}}^{1},
\end{align}
with $$G_{s}^{1,N}= \frac{2}{N\nu^2} S_{[\Delta_{s}^{N}]}^1, \quad \bar{G}_{s}^{1,N}= \frac{2}{N\nu^2} \left(S_{P(A_{s}^{N})}^1-S_{[\Delta_{s}^{N}]}^1\right) \quad  \mbox{and} \quad \hat{G}_{s}^{1,N}=  \frac{2}{N\nu^2}\bar{S}_{P(A_{s}^{N})}^1C_N.$$ 
Lemma \ref{LEMLAMUN} combined with \eqref{JX} leads to  
\begin{corollary}\label{CORJZRO}
For all $T>0$, $j(G^{1,N})\longrightarrow 0$,  as $N\longrightarrow\infty$.
\end{corollary}
For the proof of Proposition \ref{PROPSIMUN}  we will need the following lemmas
\begin{lemma}\label{GNUNS}
For any $s>0$, $$(G_{s}^{1,N}, \ s\ge0)\Longrightarrow\left( \frac{\sqrt{2}}{\nu} \sqrt{\frac{a^2+\zeta^2}{a}} B_{s}^{1}, \ s\ge0\right) {~} in{~} \mathcal{D}([0,\infty)),$$ where $B^{1}$ is a standard Brownian motion.
\end{lemma}
\begin{proof}
The result follows from Donsker's theorem (see, e.g. Theorem 14.1 page 146 in  \cite{Billingsley:2009rz}).
\end{proof}
\begin{corollary}\label{COGNUN}
The sequence $\{G^{1,N}, \ N\geq1\}$ is tight in $\mathcal{D}([0,\infty))$.
\end{corollary}
\begin{lemma}\label{GNUNBarS}
As $N\longrightarrow\infty$,  $\bar{G}^{1,N}\longrightarrow 0 \ \mbox{ in} \ \mbox{probability} \ \mbox{locally} \ \mbox{uniformly} \ \mbox{in} \ \mbox{s} $, where $\bar{G}^{1,N}$ was defined in \eqref{GUNBACHA}.
\end{lemma}
\begin{proof}
Let $\epsilon>0$ be given and $T>0$. We have 
\begin{align}\label{GBAR}
&\mathbb{P} \left(\sup_{0\leqslant s \leqslant T} \big| \bar{G}_{s}^{1,N} \big|>\epsilon \right)= \mathbb{P} \left(\sup_{0\leqslant s \leqslant T}\frac{2}{N\nu^2} \Big| S_{P(A_{s}^{N})}^1-S_{[\Delta_{s}^{N}]}^1 \Big|>\epsilon \right)\nonumber\\&\leq\mathbb{P} \left(\sup_{|s-r|\leqslant \rho,{~}0\leqslant s \leqslant T } \big| S_{[\Delta_{r}^{N}]}^1-S_{[\Delta_{s}^{N}]}^1 \big|>\epsilon \right) + \mathbb{P} \left(\sup_{0\leqslant s \leqslant T} \big| P(A_{s}^{N})- [\Delta_{s}^{N}]\big|>\left(\frac{N^2\nu^2}{2a}\right)\rho \right)\nonumber\\&\leq\mathbb{P}\left( w_{T}\big(G^{1,N},\rho\big)> \epsilon \right)+\mathbb{P} \left(\sup_{0\leqslant s \leqslant T} \big| P(A_{s}^{N})- [\Delta_{s}^{N}]\big|>\left(\frac{N^2\nu^2}{2a}\right)\rho \right).
\end{align}
Furthermore, we have 
\begin{equation*}
\frac{2a(P(A_{s}^{N})- [\Delta_{s}^{N}])}{N^2\nu^2}\longrightarrow 0 \ a.s,  \ as  \ N\longrightarrow\infty.
\end{equation*}
Indeed, it is readily seen $2a[\Delta_{s}^{N}]/N^2\nu^2\longrightarrow s \ a.s,  \ as  \ N\longrightarrow\infty$ (recall that $\Delta_{s}^{N}=  {N^2\nu^2}s/{2a}$) and we have 
\begin{equation*}
\frac{2aP(A_{s}^{N})}{N^2\nu^2}= \left(\frac{P({A}_{s}^{N})}{{A}_{s}^{N}} \right) \left(\frac{2a{A}_{s}^{N}}{N^2\nu^2} \right),
\end{equation*}
we deduce from the law of large numbers that the first factor on the right converges to $1$ a.e, as $N\longrightarrow\infty$ and from Lemma \ref{VTENDMI} that the second factor converges to $s$, as $N\longrightarrow\infty$.\\Since, moreover, for each $N$ the function $s\longrightarrow\{ {2aP(A_{s}^{N})}/{N^2\nu^2}\}$ is increasing, we deduce from the second Dini's theorem that the second term on the right in \eqref{GBAR} converges to $0$, as $N\longrightarrow\infty$. It follows that   
\begin{equation*}
\limsup_{N\rightarrow \infty} \mathbb{P} \left(\sup_{0\leqslant s \leqslant T} \big| \bar{G}_{s}^{1,N} \big|>\epsilon \right)\leq \limsup_{N\rightarrow \infty}\mathbb{P}\left( w_{T}\big(G^{1,N},\rho\big)> \epsilon \right).
\end{equation*}
Combining this inequality with \eqref{JJXX}  , we have
\begin{align*}
\limsup_{N\rightarrow \infty} \mathbb{P} \left(\sup_{0\leqslant s \leqslant T} \big| \bar{G}_{s}^{1,N} \big|>\epsilon \right)&\leq \limsup_{N\rightarrow \infty}\mathbb{P}\left( \bar{w}_{T}\big(G^{1,N},\rho\big)> \frac{\epsilon}{4} \right)+ \limsup_{N\rightarrow \infty}\mathbb{P}\left( j(G^{1,N})> \frac{\epsilon}{2} \right) \\&\leq \lim_{\rho \rightarrow0}\limsup_{N\rightarrow \infty}\mathbb{P}\left( \bar{w}_{T}\big(G^{1,N},\rho\big)> \frac{\epsilon}{4} \right)+ \limsup_{N\rightarrow \infty}\mathbb{P}\left( j(G^{1,N})> \frac{\epsilon}{2} \right).
\end{align*}
Combining Proposition \ref{ATROI}, Corollary \ref{CORJZRO}  and Corollary \ref{COGNUN}, we deduce that 
\begin{equation*}
\limsup_{N\rightarrow \infty} \mathbb{P} \left(\sup_{0\leqslant s \leqslant T} \big| \bar{G}_{s}^{1,N} \big|>\epsilon \right)=0.
\end{equation*}
The result follows
\end{proof}
\begin{lemma}\label{GNUNCha}
As $N\longrightarrow\infty$,  $\hat{G}^{1,N}\longrightarrow 0 \ \mbox{ in} \ \mbox{probability} \ \mbox{locally} \ \mbox{uniformly} \ \mbox{in} \ \mbox{s} $,  where $\hat{G}^{1,N}$ was defined in \eqref{GUNBACHA}.
\end{lemma}
\begin{proof}
We can rewrite $\hat{G}^{1,N}$  as
\begin{equation*}
\hat{G}_{s}^{1,N}=   \frac{2}{N\nu^2} \bar{S}_{[\Delta_{s}^{N}]}^1C_N +  \frac{2}{N\nu^2} \left(\bar{S}_{P(A_{s}^{N})}^1-\bar{S}_{[\Delta_{s}^{N}]}^1\right)C_N
\end{equation*}
Following the same approach as proof of the Lemma \ref{GNUNS}  and the Lemma \ref{GNUNBarS}, we have that for any $s>0$, 
$$\left( \frac{2}{N\nu^2} \bar{S}_{[\Delta_{s}^{N}]}^1, \ s\ge0\right)\Longrightarrow\left( \frac{2}{\nu} \frac{\zeta}{\sqrt{a}} B_{s}^{\circ}, \ s\ge0\right) {~} in{~} \mathcal{D}([0,\infty)),$$ 
where $B^{\circ}$ is a standard Brownian motion and that  $$\frac{2}{N\nu^2} \left(\bar{S}_{P(A_{s}^{N})}^1-\bar{S}_{[\Delta_{s}^{N}]}^1\right)\longrightarrow 0\ \mbox{ in} \ \mbox{probability} \ \mbox{locally} \ \mbox{uniformly} \ \mbox{in} \ \mbox{s}.$$ 
Since, moreover, $C_N\longrightarrow 0$ a.s (recall that $C_N={2\alpha \delta}/{(N\nu^2+2\alpha \delta)}$), the result follows readily by combining the above arguments. 
\end{proof}

We can rewrite \eqref{SECONMDEU} in the form
\begin{align}\label{SCONMDEUX}
 M_{s}^{2,N}&= -\frac{2}{N\nu^2} S_{[\hat{\Delta}_{s}^{N}]}^2 -  \frac{2}{N\nu^2} \left(S_{P^{\prime}(\hat{A}_{s}^{N})}^2-S_{[\hat{\Delta}_{s}^{N}]}^2\right) + {\Lambda}_{\hat{A}_{s}^{N}}^{2},\nonumber \\
& = G_{s}^{2,N}- \bar{G}_{s}^{2,N}+ {\Lambda}_{\hat{A}_{s}^{N}}^{2},
\end{align}
with $$G_{s}^{2,N}=- \frac{2}{N\nu^2} S_{[\hat{\Delta}_{s}^{N}]}^2\quad  \mbox{and} \quad \bar{G}_{s}^{2,N}= \frac{2}{N\nu^2} \left(S_{P^{\prime}(\hat{A}_{s}^{N})}^2-S_{[\hat{\Delta}_{s}^{N}]}^2\right).$$ 
Similarly as above, we deduce the following results
\begin{lemma}
For any $s>0$, $$(G_{s}^{2,N}, \ s\ge0)\Longrightarrow\left( \frac{\sqrt{2}}{\nu} B_{s}^{2}, \ s\ge0\right) {~} in{~} \mathcal{D}([0,\infty)),$$ where $B^{2}$ is a standard Brownian motion.
\end{lemma}
\begin{lemma}
As $N\longrightarrow\infty$,  $\bar{G}^{2,N}\longrightarrow 0 \ \mbox{ in} \ \mbox{probability} \ \mbox{locally} \ \mbox{uniformly} \ \mbox{in} \ \mbox{s} $.
\end{lemma}
Since, the sequences $(\tilde{\Theta}_{k} \Xi_{k}- {\Theta}_{k} \tilde{\Xi}_k)_{k\geq1}$ and $(\tilde{\Xi}^{\prime}_k)_{k\geq1}$ are independent, the processes $\{G^{1,N}, \ N\geq1\}$ and $\{G^{2,N}, \ N\geq1\}$ are also independent. Consequently the assertion of Proposition \ref{PROPSIMUN}  is now immediate by combining the above convergence results.

Let us state a basic result for counting processes, which will be useful in the sequel.  For this, let $Q$ be a counting process with stochastic intensity  $(\lambda(t), t\geq0)$. Let $\mathcal{F}=(\mathcal{F}_{t}, t\geq0)$ be the filtration generated by $Q$. Let  $(S_{k}, k\geq 1)$ be the successive jump times of $Q$, and suppose that  $$Q_{t}- \int_{0}^{t} \lambda(r)dr$$ is an $\mathcal{F}$-martingale.
\begin{lemma}\label{EXPOSTAN} 
The sequence 
\begin{eqnarray*}
{\left(\int_{S_{k-1}}^{S_{k}}\lambda(r)dr \right)}_{k\geq1}
\end{eqnarray*}
is a sequence of i.i.d standard exponential random variables.
\end{lemma}
\begin{proof}
Let $$ \nu_{t}= \int_{0}^{t} \lambda(r)dr$$ and let $ \tau$ be the inverse of $\nu$, that is,  $$\tau_{u}=  \inf \{ t >0: \nu_{t} > u\};$$ see Exercice 5.13 of Chapter I in $\cite{ccinlar2011probability}$. Then, $\tau$ is right-continuous and strictly increasing, and $\nu_{\tau_{u}}= u$ by the continuity of $\nu$. Clearly, $(Q_{\tau_{u}})$ is adapted to the filtration $(\mathcal{F}_{\tau_{u}})$ and is again a counting process. Since $Q - \nu$ is assumed to be an $\mathcal{F}$-martingale and since $\tau_{u}$ is a stopping time, we deduce from Doob's optional stopping theorem that the process $(Q_{\tau_{u}}- u)$ is an $(\mathcal{F}_{\tau_{u}})$-martingale. By Proposition 6.13 in $\cite{ccinlar2011probability}$, the process $Q^{\circ}_{u}=Q_{\tau_{u}}$ is a standard Poisson process. It follows that 
\begin{eqnarray*}
\int_{S_{k-1}}^{S_{k}}\lambda(r)dr= \nu(S_{k})- \nu(S_{k-1})=T_{k}- T_{k-1},
\end{eqnarray*}
where $T_{k}$ is the $k$th jump time of $Q^{\circ}$. Consequently the $T_{k}- T_{k-1}$ are independent and identically distributed standard exponential random variables.
\end{proof}
\\To ease the reading, we rewrite \eqref{DEUDFHN} in the following form 
\begin{equation}\label{TROIDEFH}
 H_{s}^{N,\Gamma}+\frac{V_{s}^{N}}{N\nu^2} = \frac{1}{N\nu^2}+ M_{s}^{1,N}-M_{s}^{2,N} +\Psi_{s}^{N} - Q_{s}^{N,+,2} + F^{N}(s)+ \frac{1}{2}\left(L_{s}^{N,\Gamma}(0)- L_{0^{+}}^{N,\Gamma}(0)\right)  - \frac{1}{2} L_{s}^{N,\Gamma}(\Gamma^{-})
\end{equation}
where  $$F^{N}(s)=  \frac{4\alpha \delta}{a\nu^2}\int_{0}^{s}\mathbf{1}_{\{V_{r}^{N}=-1\}}{{\Theta}_{\widetilde{P}_{r}^{N,+}+1}}dr- \frac{4\beta \delta}{\nu^2} \int_{0}^{s}\mathbf{1}_{\{V_{r}^{N}=+1\}}dr.$$
In the following, we give useful properties of the sequence of processes $(F^{N}, \ N\ge1)$. 
\begin{lemma}
For any $s>0$, 
$$F^{N}(s) \longrightarrow  F(s){~}in{~}probability, {~}as{~}N\longrightarrow \infty,$$
where  \quad $F(s) = \frac{2(\alpha-\beta)}{\kappa^2}s$.
\end{lemma}
\begin{proof}
We have 
\begin{align*}
F^{N}(s)&= \frac{4\alpha \delta}{a\nu^2}\int_{0}^{s}\mathbf{1}_{\{V_{r}^{N}=-1\}}{{\Theta}_{\widetilde{P}_{r}^{N,+}+1}}dr- \frac{4\beta \delta}{\nu^2} \int_{0}^{s}\mathbf{1}_{\{V_{r}^{N}=+1\}}dr\\
&=  \frac{4\alpha \delta}{a\nu^2}\sum_{k=1}^{{\widetilde{P}_{s}^{N,+}}}{\Theta}_{k}\int_{S_{k-1}^{N,+}}^{S_{k}^{N,+}}\mathbf{1}_{\{V_{r}^{N}=-1\}}dr + \varsigma_s^N- \frac{4\beta \delta}{\nu^2} \int_{0}^{s}\mathbf{1}_{\{V_{r}^{N}=+1\}}dr
\end{align*}
where $\varsigma_s^N$ describes the boundary effects at the points $s$, tending to $0$ as $N$ goes to ${\infty}$. Indeed, we have
$$\varsigma_s^N\leq C{\Theta}_{{\widetilde{P}_{s}^{N,+}}}\int_{S_{{\widetilde{P}_{s}^{N,+}}}^{N,+}}^{S_{{\widetilde{P}_{s}^{N,+}}+1}^{N,+}}\mathbf{1}_{\{V_{r}^{N}=-1\}}dr.$$
For $k\ge  1$ define, with $\lambda^N(r)= a^{-1}({N^2\nu^2+2N\alpha \delta})\mathbf{1}_{\{V_{r^{-}}^{N}=-1\}}$
\begin{equation}\label{XHI}
 \xi_{k}^{N}=\int_{S_{k-1}^{N,+}}^{S_{k}^{N,+}}\lambda^{N}(r)dr.
 \end{equation}
 Hence we have 
$$\varsigma_s^N\leq \left(\frac{aC}{N^2\nu^2+2N\alpha \delta}\right) {\Theta}_{{\widetilde{P}_{s}^{N,+}}} \xi_{{\widetilde{P}_{s}^{N,+}}+1}^{N}.$$ It follows readily from Lemma \ref{EXPOSTAN}  that $$\mathbb{E}(\varsigma_s^N)^2\longrightarrow 0{~} \mbox{as} {~}N\longrightarrow\infty.$$
However, from  Lemma \ref{EXPOSTAN}, \eqref{XHI} and  Remark \ref{PETPPRIM}, we deduce that
\begin{equation*}
F^{N}(s)= \left(\frac{4\alpha \delta A_{s}^{N}}{\nu^2(N^2\nu^2+2N\alpha\delta)} \right) \left(\frac{P(A_{s}^{N})}{A_{s}^{N}}\right) \left( \frac{1}{P(A_{s}^{N})}\sum_{k=1}^{P(A_{s}^{N})}{\Theta}_{k}\xi_{k}^{N}\right)  -  \frac{4\beta \delta}{\nu^2} \int_{0}^{s}\mathbf{1}_{\{V_{r}^{N}=+1\}}dr + \varsigma_s^N.
\end{equation*}
From Lemma \ref{VTENDMI}, we deduce that the first factor of the first term on the right converges to ${2\alpha}s/{a\kappa^2}$ as $N\longrightarrow\infty$. We have  from the law of large numbers that the second factor converges to $1$ a.e, as $N\longrightarrow\infty$. Moreover it follows from  the strong law of large numbers that the third factor converges to $a$  in probability, as $N\longrightarrow\infty$. Now combining the above arguments, we deduce that
$$F^{N}(s) \longrightarrow \frac{2(\alpha-\beta)}{\kappa^2}s{~}in{~}probability, {~}as{~}N\longrightarrow \infty.$$
\end{proof}
\\In addition 
\begin{lemma}\label{TENFN}
The sequence $\{F^{N},  N\ge  1\}$ is tight in $\mathcal{C}([0,\infty))$.
\end{lemma}
\begin{proof}
We have
\begin{align*}
F^{N}(s)&=  \frac{4\alpha \delta}{a\nu^2}\int_{0}^{s}\mathbf{1}_{\{V_{r}^{N}=-1\}}{{\Theta}_{\widetilde{P}_{r}^{N,+}+1}}dr -  \frac{4\beta \delta}{\nu^2} \int_{0}^{s}\mathbf{1}_{\{V_{r}^{N}=+1\}}dr \\&=b_{s}^N-d_{s}^N.
\end{align*}
The sequence $\{b^{N},  N\ge  1\}$ is tight in $\mathcal{C}([0,\infty))$. Indeed, For all $0< s< t$
\begin{align*}
\E\left( \sup_{s\leq t\leq s+\rho} |b_{t}^{N}-b_{s}^{N}|^2\right)&= \E \left( \sup_{s\leq t\leq s+\rho} \Big|\frac{4\alpha \delta}{a\nu^2} \int_{s}^{t}\mathbf{1}_{\{V_{r}^{N}=-1\}} {\Theta}_{P_{r}^{N,+}+1}dr \Big|^2\right)  \\
& \leqslant\left(\frac{4\alpha \delta}{a\nu^2} \right)^2  \E \left(\int_{s}^{s+\rho}|{\Theta}_{P_{r}^{N,+}+1}|dr \right)^2\\ 
&\leqslant \rho\left(\frac{4\alpha \delta}{a\nu^2} \right)^2  \E \left(\int_{s}^{s+\rho}{\Theta}_{P_{r}^{N,+}+1}^2dr \right)\\
& \leqslant \left(\frac{4\alpha \delta}{a\nu^2} \right)^2 (\zeta^{2}+ a^2)\rho^2 .
\end{align*}
We obtain similarly, the tightness of $\{d^{N},  N\ge  1\}$ in $\mathcal{C}([0,\infty))$. Consequently $\{ F^{N},  N\ge  1\}$ is tight in $\mathcal{C}([0,\infty))$.
\end{proof}
\begin{corollary}
$F^{N}\longrightarrow F$ in probability in $\mathcal{C}([0,\infty),{\mathbb{R}}_{+})$.
\end{corollary}
Let us rewrite \eqref{TROIDEFH} in the form 
\begin{equation}\label{QUATRHN}
H_{s}^{N, \Gamma}= R_{s}^N+\frac{1}{2} L_{s}^{N, \Gamma}(0) -\frac{1}{2} L_{s}^{N, \Gamma}( \Gamma^{-}),
\end{equation}
where 
\begin{equation}\label{OURNS}
 R_{s}^N= \frac{1}{N\nu^2}-\frac{V_{s}^{N}}{N\nu^2}+ M_{s}^{1,N}-M_{s}^{2,N} +\Psi_{s}^N + F^N(s) - \frac{1}{2} L_{0^{+}}^{N}(0)- Q_{s}^{N,+,2},   \quad s\ge 0.
\end{equation}
We have proved in particular
\begin{lemma}
\label{TENSIORN} 
The sequence $\{R^{N},  N\ge  1\}$ is tight in $\mathcal{D}([0,\infty))$.
\end{lemma}
\begin{proof}
We may rewrite \eqref{OURNS} as
\begin{equation}\label{RNPLUQ}
 R_{s}^N+\frac{V_{s}^{N}}{N\nu^2}+ \frac{1}{2} L_{0^{+}}^{N}(0)+ Q_{s}^{N,+,2} - \frac{1}{N\nu^2}= M_{s}^{1,N}-M_{s}^{2,N} +\Psi_{s}^N + F^N(s) .
\end{equation}
Tightness of the right-hand side of \eqref{RNPLUQ} follows from Proposition \ref{PROPSIMUN}, Lemma \ref{TENFN}  and Proposition \ref{AUN} . From \eqref{TL}, it is easily checked that $L_{0^{+}}^{N}(0)= 4/N\kappa^2\delta$. Since, moreover $N^{-1}V_{s}^{N}\longrightarrow0$ a.s uniformly with respect to $s$ and $Q_{s}^{N,+,2}\longrightarrow0$ in probability, locally uniformly in $s$, the sequence $\{R^{N}, \ N\ge  1\}$ is tight in $\mathcal{D}([0,\infty))$.
\end{proof}

In what follows, we investigate the tightness property of $H^{N,\Gamma}$ by help of Lemma \ref{TENSIORN} and the function $j(.)$ was defined in \eqref{JX}. We have
\begin{proposition}\label{HNTENSIO}
For any $\Gamma>0$, the sequence $\{H^{N,\Gamma}, \ N\geq1\}$ is tight  in $\mathcal{C}([0,\infty))$.
\end{proposition}
\begin{proof}
We show that the sequence $\{H^{N, \Gamma}, \ N\geq1\}$ satisfies the conditions of Proposition \ref{AZERO}. Condition $(i)$ follows easily from $H_{0}^{N, \Gamma}=0$. In order to verify condition $(ii)$, we will show that for each $\epsilon \geq0$, $$\lim_{\rho \rightarrow0} \limsup_{N\rightarrow \infty} \mathbb{P}(w_{T}(H^{N, \Gamma},\rho)\geq \epsilon)=0.$$
Indeed,  let $\epsilon>0$ be given and $T>0$. Since $L^{N, \Gamma}(0)$ $($resp. $L^{N, \Gamma}( \Gamma^{-}))$ increases only when $H_{s}^{N, \Gamma}=0$ $($resp. when $H_{s}^{N, \Gamma}=\Gamma$ $)$, it is not hard to conclude from \eqref{QUATRHN} that  for any $0<d<1$, 
\begin{equation*}
\left\{\sup_{|s-t|\leqslant \rho,{~}0\leqslant r,s \leqslant T } \left|{H}_{s}^{N, \Gamma}-{H}_{r}^{N,  \Gamma}\right|> \epsilon  \right\} \subset \left\{\sup_{|s-t|\leqslant \rho,{~}0\leqslant r,s \leqslant T } \left|{R}_{s}^{N}-{R}_{r}^{N}\right|>d \epsilon  \right\} 
\end{equation*}
However, from \eqref{JJXX}  we have 
\begin{equation*}
\mathbb{P}\left(\sup_{|s-t|\leqslant \rho,{~}0\leqslant r,s \leqslant T } \left|{R}_{s}^{N}-{R}_{r}^{N}\right|>d \epsilon   \right) \leq \mathbb{P}\left( \bar{w}_{R^N,T}(\rho)>\frac{d\epsilon}{4} \right)+ \mathbb{P}\left( j(R^N)>\frac{d\epsilon}{2}  \right).
\end{equation*}
Consequently
\begin{equation*}
\mathbb{P}\left(\sup_{|s-t|\leqslant \rho,{~}0\leqslant r,s \leqslant T } \left|{H}_{s}^{N, \Gamma}-{H}_{r}^{N, \Gamma}\right|> \epsilon  \right) \leq \mathbb{P}\left( \bar{w}_{R^N,T}(\rho)>\frac{d\epsilon}{4} \right)+ \mathbb{P}\left( j(R^N)>\frac{d\epsilon}{2}  \right).
\end{equation*}
Combining this inequality with Lemma \ref{TENSIORN} and the fact that $j(R^{N})\leq \frac{C}{N}$, we deduce that 
$$\lim_{\rho \rightarrow0} \limsup_{N\rightarrow \infty} \mathbb{P}(w_{T}(H^{N, \Gamma},\rho)\geq \epsilon)=0.$$ The result follows.
\end{proof}

We can now establish 
\begin{lemma}\label{LNETLNGAMA}
The sequence $\{L^{N,\Gamma}(0),\ N\geq1\}$ $($resp. $\{L^{N,\Gamma}(\Gamma^{-}),\ N\geq1\})$ is tight in $\mathcal{D}([0,\infty))$, the limit  $L^{\Gamma}(0)$ $($resp. $L^{\Gamma}(\Gamma^{-}))$  of any converging subsequence being continuous and increasing.
\end{lemma}
\begin{proof}
Let us rewrite \eqref{QUATRHN} in the form 
\begin{equation*}
H_{s}^{N, \Gamma}= R_{s}^N+\frac{1}{2} U_{s}^{N}
\end{equation*}
where $U_{s}^{N}= L_{s}^{N, \Gamma}(0) - L_{s}^{N, \Gamma}( \Gamma^{-})$.
It then follows from Proposition \ref{AUN}, Lemma \ref{TENSIORN} and Proposition \ref{HNTENSIO} that  $\{U^{N}, N\geq1\}$ is tight in $\mathcal{D}([0,\infty))$. We now show that the sequence $\{L^{N,\Gamma}(0),\ N\geq1\}$ satisfies the condition of Corollary \ref{Aplus}. Indeed,  let $\epsilon>0$ be given and $T>0$. Since $L^{N, \Gamma}(0)$ $($resp. $L^{N, \Gamma}( \Gamma^{-}))$ increases only on the set of time when $H_{s}^{N, \Gamma}=0$ $($resp. $H_{s}^{N, \Gamma}=\Gamma$ $)$, we notice that 
\begin{equation*}
\sup_{|s-t|\leqslant \rho} \left|{L}_{s}^{N, \Gamma}(0)-{L}_{r}^{N,  \Gamma}(0) \right| \leq \sup_{|s-t|\leqslant \rho} \left|{U}_{s}^{N}-{U}_{r}^{N}\right|  \quad \mbox{unless}  \quad \sup_{|s-t|\leqslant \rho} \left|{H}_{s}^{N, \Gamma}-{H}_{r}^{N,  \Gamma}\right|>\frac{\Gamma}{2}.
\end{equation*}
It then follows from \eqref{ModulC} and \eqref{JJXX}  that 
\begin{align*}
\mathbb{P}\Big(w_{T}\big(L^{N, \Gamma}(0),\rho\big)\geq \epsilon\Big)&\leq \mathbb{P}\Big(w_{T}\big(U^{N},\rho\big)\geq \epsilon\Big) +  \mathbb{P}\Big(w_{T}\big(H^{N, \Gamma},\rho\big)>\frac{\Gamma}{2}\Big)\\
&\leq  \mathbb{P}\Big(\bar{w}_{T}\big(U^{N},\rho\big)\geq \frac{\epsilon}{2}\Big)+  \mathbb{P}\Big(j(U^{N})\geq \frac{\epsilon}{2}\Big)+ \mathbb{P}\Big(w_{T}\big(H^{N, \Gamma},\rho\big)>\frac{\Gamma}{2}\Big). 
\end{align*}
The assertion of Corollary \ref{Aplus} is now immediate by combining Proposition \ref{HNTENSIO},  tightness in $\mathcal{D}([0,\infty))$ of $\{U^{N}, N\geq1\}$ and the fact that $j(U^{N})\leq \frac{C}{N}$. We deduce that the sequence $\{L^{N,\Gamma}(0),\ N\geq1\}$ is tight in $\mathcal{D}([0,\infty))$. Now we show that the limit $K$ of any converging subsequence is continuous and increasing.  To this end, for each $l\geq1$,  we define the function $f_{l} : {\mathbb{R}}_{+}\rightarrow [0,1]$ by $f_{l}(x)={(1-lx)}^{+}$. We have that for each $N,l\geq1$, $s>0$, since $L^{N,\Gamma}(0)$ increases only when $H^{N,\Gamma}=0$, 
\begin{align*}
\mathbb{E} \left( \int_{0}^{s}f_{l}(H_{r}^{N,\Gamma})dL_{r}^{N,\Gamma}(0)- L_{s}^{N,\Gamma}(0)\right)\geq 0.
\end{align*}
Thanks to Lemma \eqref{LDEU}, we can take the limit in this last inequality as $N\rightarrow\infty$, yielding
\begin{align*}
\mathbb{E} \left( \int_{0}^{s}f_{l}(H_{r}^{\Gamma})dK_{r}- K_{s}\right)\geq 0.
\end{align*}
Then taking the limit as $l\rightarrow +\infty$ yields
\begin{align*}
\mathbb{E} \left( \int_{0}^{s}\mathbf{1}_{\{H_{r}^{\Gamma}=0\}}dK_{r}- K_{s}\right)\geq 0.
\end{align*}
But the random variable under the expectation is clearly nonpositive, hence it is zero
a.s., in other words
\begin{align*}
 \int_{0}^{s}\mathbf{1}_{\{H_{r}^{\Gamma}=0\}}dK_{r} =K_{s},{~} p.s,{~~} \forall s\geq0.
\end{align*}
which means that the process $K$ increases only when $H_{r}^{\Gamma}=0$. From the occupation times formula
\begin{align*}
 \int_{0}^{s}g(H_{r}^{\Gamma})dr=  \int_{0}^{s}g(t)L_{s}^{\Gamma}(t)dt
\end{align*}
applied to the function $g(h)= \mathbf{1}_{\{h=0\}}$, we deduce that the time spent by the process $H^{\Gamma}$ at 0 has a.s. zero Lebesgue measure. Consequently
\begin{align*}
 \int_{0}^{s}\mathbf{1}_{\{H_{r}^{\Gamma}=0\}}dB_{r}\equiv0 {~~} a.s.
\end{align*} 
hence a.s.
\begin{align*}
B_{s}=\int_{0}^{s}\mathbf{1}_{\{H_{r}^{\Gamma}>0\}}dB_{r} {~~} \forall s\geq0.
\end{align*} 
It then follows from Tanaka’s formula applied to the process $H^{\Gamma}$ and the function $h\rightarrow h^{+}$ that $K=L^{\Gamma}(0)$. The continuity of $K$ follows from Corollary \ref{Aplus}. \\ Following the same approach as $L^{N,\Gamma}(0)$, we have $\{L^{N,\Gamma}(\Gamma^{-}),\ N\geq1\})$ is tight in $\mathcal{D}([0,\infty))$, the limit $L^{\Gamma}(\Gamma^{-})$ of any converging subsequence being continuous and increasing. In other words $L^{\Gamma}(\Gamma^{-})$ is the local time of $H^{\Gamma}$ at  level $\Gamma^{-}$.
\end{proof}

An immediate consequence of these results is
\begin{proposition}
For each $\Gamma>0$,
\begin{eqnarray*}
\bigg(H^{N,\Gamma}, F^N, \Psi^{N}+M^{1,N},M^{2,N},L^{N,\Gamma}(0),L^{N, \Gamma}( \Gamma^{-})\bigg)&\Rightarrow \bigg(H^{\Gamma}, F, \frac{\sqrt{2}}{\nu} \sqrt{\frac{a^2+\zeta^2}{a}}  B^{1},  \frac{\sqrt{2}}{\nu} B^{2},L^{\Gamma}(0), L^{\Gamma}( \Gamma^{-})\bigg)\\& in{~} {(\mathcal{C}([0,\infty)))}^{2}\times{(\mathcal{D}([0,\infty)))}^{4}
\end{eqnarray*}
as $N\rightarrow \infty$, where $F$, $B^{1}$ and $B^{2}$ are as above, $L^{\Gamma}(0)$ $($resp. $L^{\Gamma}(\Gamma^{-})$  is the local time of $H^{\Gamma}$ at  level $0$  $($resp. at  level $\Gamma^{-})$ and $H^{\Gamma}$ is the unique weak solution of SDE 
\begin{equation}\label{HSGAMA}
H_{s}^{\Gamma}=  \frac{2(\alpha-\beta)}{\kappa^2}s+ \frac{2}{\kappa} B_{s}+ \frac{1}{2}L_{s}^{\Gamma}(0)- \frac{1}{2}L_{s}^{\Gamma}(\Gamma^{-}), \quad s\geq 0,
\end{equation}
i.e. $H^{\Gamma}$ equals $2/\kappa$ multiplied by Brownian motion with drift $(\alpha-\beta)s / \kappa$, reflected in the interval $[0,\Gamma]$.
\end{proposition}

We are now ready to state the main result.
\begin{theorem}
The following holds
 \begin{eqnarray*}
H^{N,{\Gamma}}\Longrightarrow H^{\Gamma} {~} {~}in{~} \mathcal{C}([0,\infty)),{~} as{~} N \longrightarrow\infty.
\end{eqnarray*}
\end{theorem}
\begin{proof}
Equation \eqref{HSGAMA} follows by taking the limit in \eqref{TROIDEFH} combined with the above results. It is plain that $H^{\Gamma}$, being a limit  (along a subsequence) of $H^{N,\Gamma}$, takes values in $[0,\Gamma]$. The fact that $L^{\Gamma}(0)$ $($resp. $L^{\Gamma}( \Gamma^{-}))$ is continuous and increasing, and increases only on the set of time when $H_{s}^{\Gamma}=0$ $($resp. $H_{s}^{\Gamma}=\Gamma$ $)$ proves that $\frac{\kappa}{2	}H^{\Gamma}$ is a Brownian motion with drift $(\alpha-\beta)s / \kappa$, reflected in $[0,\Gamma]$, which characterizes its law. We can refer e.g. to the formulation of reflected SDEs in \cite{lions1984stochastic}.
\end{proof}
\subsection{The subcritical case} \label{II.4}
We now want to establish a similar statement for weak convergence of the height process  in the subcritical case $(i.e \ \alpha < \beta)$ without reflecting the process $H^{N,{\Gamma}}$. In other words, in the subcritical case, we can choose $\Gamma=+\infty$, which simplifies the above construction. The two main difficulties are the need for a new bound for $\mathbb{E} ({T}_{1}^N)$ (see Lemma \ref{ESPTUN}), and a bound for $\mathbb{E}\left(\sup_{0\leq s\leq T}{H}_{s}^{N}\right)$, since we cannot use Remark \ref{HNINFGAMA} anymore. Now we notice that in the subcritical case $(i.e \ \alpha < \beta)$, the constant $C(\Gamma)$ defined in \eqref{ESZNT} in the proof of Lemma \ref{ESTAYTO}, is bounded by  $1/(\beta- \alpha)$ for all $\Gamma>0$. In that case, we can choose $\Gamma=+\infty$. Consequently, an easy adaptation gives a result similar to Lemma \ref{ESPTUN}. In the subcritical case, the equation \eqref{TROIDEFH} takes the following form
\begin{align}\label{HNSOUCRI}
 H_{s}^{N}+\frac{V_{s}^{N}}{N\nu^2}&= \frac{1}{N\nu^2}+   {\tilde{M}}_{s}^{N}+\Psi_{s}^{N} - Q_{s}^{N,+,2} + F^{N}(s)+ \frac{1}{2}(L_{s}^{N}(0)- L_{0^{+}}^{N}(0)),
\end{align}
where
\begin{equation*}
{\tilde{M}}_{s}^{N}=M_{s}^{1,N}-M_{s}^{2,N}.
\end{equation*}
Since ${M}_{s}^{1,N}$ and ${M}_{s}^{2,N}$ are two orthogonal martingales, we deduce from \eqref{PREMUN} and \eqref{PREMDEU} that  
\begin{equation}\label{CROCHET}
\langle{{\tilde{M}}^{N}\rangle}_{s}= \int_{0}^{s} \left( {\Theta_{{\widetilde{P}_{r}^{N,+}+1}}^2}\Big(\frac{4\nu^2N+8\alpha\delta}{aN\nu^4}\Big)\mathbf{1}_{\{V_{r}^{N}=-1\}}+ \Big(\frac{4\nu^2N+8\beta\delta}{N\nu^4}\Big)\mathbf{1}_{\{V_{r}^{N}=1\}}\right)dr. 
\end{equation}
From \eqref{CROCHET} we deduce that $\{{\tilde{M}}_{s}^{N}, s\ge 0\}$ is in fact a martingale. Recall That $\Psi^{N}$ was given by
\begin{equation}\label{SECONPSI}
 \Psi_{s}^{N}= \frac{2}{N\nu^2+2\alpha \delta} \bigg(\sum_{k=0}^{P(A_{s}^{N})} { \tilde{\Theta}_{k+1}} \Xi_k-{ \tilde{\Theta}_{P(A_{s}^{N})+1}}({T}^+(A_{s}^{N}) - A_{s}^{N})\bigg),
\end{equation}
see \eqref{PRPSIN}. We now set for $\ell \ge1$ 
 \begin{equation}\label{DEFPHILN}
 \Phi_{\ell}^{N}= \frac{2}{N\nu^2+2\alpha \delta} \sum_{k=0}^{\ell} { \tilde{\Theta}_{k+1}} \Xi_k.
\end{equation} 
We will need the following lemmas
\begin{lemma}\label{LEMPHI} 
There exist a constant $C>0$ such that for all $T>0$, 
\begin{equation*}
 \mathbb{E} {\left(  \Phi_{P(\bar{A}_{T}^{N})}^{N}\right)}^2 \le CT,
 \end{equation*}
(recall that $\bar{A}_{s}^{N}={(N^2\nu^2+2N\alpha \delta)}s/{2a}$).
\end{lemma}
\begin{proof}
Since the random variables $P(\bar{A}_{T}^{N})$, ${ \tilde{\Theta}_{1}}$ and $\Xi_1$ are mutually independent, we have from \eqref{DEFPHILN}  that 
\begin{equation*}
 {\left(  \Phi_{P(\bar{A}_{T}^{N})}^{N}\right)}^2 =  \frac{4}{(N\nu^2+2\alpha \delta)^2}\bigg[ \sum_{k=0}^{P(\bar{A}_{s}^{N})} \left({ \tilde{\Theta}_{k+1}} \Xi_k\right)^2 +  \sum_{1\le i \neq k \le P( \bar{A}_{s}^{N})} { \tilde{\Theta}_{k+1}} { \tilde{\Theta}_{i+1}}  \Xi_k \Xi_i  \bigg]
\end{equation*}
Hence tacking expectation in both side, we deduce that 
\begin{align*}
\mathbb{E} \left(  \Phi_{P(\bar{A}_{T}^{N})}\right)^2&=  \frac{4}{(N\nu^2+2\alpha \delta)^2}\bigg[ \mathbb{E} \left(  P(\bar{A}_{T}^{N})\right)  \left( \mathbb{E} \left( \tilde{\Theta}_{1}\right)^2\right) \left(\mathbb{E} \left(  \Xi_1\right)^2\right) \nonumber\\\\&+ \  \mathbb{E} \Big( P(\bar{A}_{T}^{N}) (P(\bar{A}_{T}^{N}) -1)\Big) \Big(  \mathbb{E} \left( \tilde{\Theta}_{1}\right)\Big)^2 \Big( \mathbb{E} \left(  \Xi_1\right)\Big)^2   \bigg].
\end{align*}
\\The second term on the right is zero because the random variable $\tilde{\Theta}_{1}$ is centered and since  $$\mathbb{E} {\left(  P(\bar{A}_{T}^{N})\right)}\le CTN^2,$$ we deduce that 
\begin{equation*}
 \mathbb{E} {\left(  \Phi_{P(\bar{A}_{T}^{N})}^{N}\right)}^2 \le CT.
\end{equation*}
\end{proof}
\begin{lemma}\label{ESPSISN} 
There exist a constant $C>0$ such that for all $T>0$, 
\begin{equation*}
\mathbb{E}\left(\sup_{0\leq s\leq T}|{\Psi}_{s}^{N}|\right)\leq C(1+T).
\end{equation*}
\end{lemma}
\begin{proof}
From \eqref{SECONPSI} and \eqref{DEFPHILN},  it follows that 
 \begin{equation*}
\Psi_{s}^{N}= \Phi_{P(A_{s}^{N})}^{N}-\frac{2}{N\nu^2+2\alpha \delta}{ \tilde{\Theta}_{P(A_{s}^{N})+1}}({T}^+(A_{s}^{N}) - A_{s}^{N}).
\end{equation*}
 However, noting that ${ \tilde{\Theta}_{k}}^{-}\leqslant a,$ where ${\tilde{\Theta}_{k}}^{-}=\sup\{-\tilde{\Theta}_{k},0\}$,  
\begin{align*}
\quad \sup_{0\leq s\leq T} \left|{\Psi}_{s}^{N}\right|&\leq \sup_{0\leq k\leq P({A}_{T}^{N})}  \left| \Phi_{k}^{N}\right|+\frac{2a}{N\nu^2+2\alpha \delta} \sup_{0\leq s\leq T}({T}^+({A}_{s}^{N}) - {A}_{s}^{N}).
\end{align*}
It follows that 
\begin{align*}
\mathbb{E}\left(\sup_{0\leq s\leq T} \left|{\Psi}_{s}^{N}\right| \right)&\leq \left[\mathbb{E}\left( \sup_{0\leq k\leq P({A}_{T}^{N})}  \left| \Phi_{k}^{N}\right|\right)^2\right]^{\frac{1}{2}}+\frac{2a}{N\nu^2+2\alpha \delta} \mathbb{E}\left( \sup_{0\leq s\leq T}({T}^+(A_s^N) - A_s^N) \right)\\&\leq   \left[\mathbb{E}\left( \sup_{0\leq k\leq P(\bar{A}_{T}^{N})}  \left| \Phi_{k}^{N}\right|\right)^2\right]^{\frac{1}{2}}+\frac{2a}{N\nu^2+2\alpha \delta} \mathbb{E}\left( \sup_{0\leq s\leq T}({T}^+({A}_{s}^{N}) - {A}_{s}^{N}) \right). 
\end{align*}
It is easy to check that $( \Phi_{k}^N, \ k\ge0)$ is a discrete-time martingale. Moreover, note that $P(\bar{A}_{T}^{N})$ is a stopping time. Hence, from Doob's inequality we have 
\begin{equation*}
\mathbb{E}\left(\sup_{0\leq s\leq T} \left|{\Psi}_{s}^{N}\right| \right) \leq 2 \left[\mathbb{E}\left(  \Phi_{P(\bar{A}_{T}^{N})}^{N}\right)^2\right]^{\frac{1}{2}}+ C
\end{equation*}
The result now follows readily from Lemma \ref{LEMPHI}. 
\end{proof}

Now we deduce the following basic estimate for $H^N$
\begin{lemma} 
 There exist a constant $C>0$ such that for all $T>0$,
\begin{equation*}
\mathbb{E}\left(\sup_{0\leq s\leq T}{H}_{s}^{N}\right)\leq C(1+T).
\end{equation*}
\end{lemma}
\begin{proof}
Let us rewrite \eqref{HNSOUCRI} in the form 
\begin{equation*}
H_{s}^{N}= W_{s}^{N} +\frac{1}{2}L_{s}^{N} (0),
\end{equation*}
where 
\begin{equation}\label{WNSEGA}
W_{s}^{N}= \frac{1}{N\nu^2} -\frac{V_{s}^{N}}{N\nu^2}+  {\tilde{M}}_{s}^{N}+\Psi_{s}^N  + F^{N}(s)- \frac{1}{2} L_{0^{+}}^{N}(0)- Q_{s}^{N,+,2}.
\end{equation}
Set $$s_N= \sup \left\{ 0\le r\le s;  \ L_{r}^{N}(0)- L_{r^{-}}^{N}(0)>0\right\},$$
then the fact that $L^{N}(0)$ is  increasing, and increases only on the set of time when $H_{s}^{N}=0$ proves that  $H_{s_N}^{N}=0$ and $L_{s}^{N}(0)= L_{s_N}^{N}(0)$. It follows that  
$$H_{s}^{N}= W_{s}^{N}- W_{s_N}^{N}.$$ Hence $$H_{s}^{N} \le \sup_{0\le r \le s} [W_{s}^{N}- W_{r}^{N}],$$ this implies
\begin{equation}\label{SUPHNS}
\sup_{0\le s \le T} H_{s}^{N} \le  \sup_{0\le r \le s \le T} [W_{s}^{N}- W_{r}^{N}] \le 2 \sup_{0\le s \le T} |W_{s}^{N}|.
\end{equation}
However,  from \eqref{TL}, it is easily checked that $L_{0^{+}}^{N}(0)= 4/N\kappa^2\delta$. Since, moreover $(Q_{s}^{N,+,2}, s\ge0)$ is a increasing process with values in ${\mathbb{R}}_{+}$ see \eqref{QDEPLUS}, we have from \eqref{WNSEGA} that 
\begin{align*}
\sup_{0\le s \le T} | W_{s}^{N} | &\le  \sup_{0\leq s\leq T} |{\tilde{M}}_{s}^{N}| +\sup_{0\leq s\leq T}|{\Psi}_{s}^{N}| + \frac{4\alpha \delta}{a\nu^2}\int_{0}^{T}\mathbf{1}_{\{V_{r}^{N}=-1\}}{{\Theta}_{\widetilde{P}_{r}^{N,+}+1}}dr  \nonumber\\
&+  \frac{4\beta \delta}{\nu^2} \int_{0}^{T}\mathbf{1}_{\{V_{r}^{N}=+1\}}dr + Q_{T}^{N,+,2}+ C.
\end{align*}
Combining this inequality with  \eqref{SUPHNS}, we deduce that 
\begin{align*}
\sup_{0\leq s\leq T} H_{s}^{N}&\leq 2\sup_{0\leq s\leq T} |{\tilde{M}}_{s}^{N}| +2\sup_{0\leq s\leq T}|{\Psi}_{s}^{N}| + \frac{8\alpha \delta}{a\nu^2}\int_{0}^{T}\mathbf{1}_{\{V_{r}^{N}=-1\}}{{\Theta}_{\widetilde{P}_{r}^{N,+}+1}}dr\nonumber\\
&+  \frac{8\beta \delta}{\nu^2} \int_{0}^{T}\mathbf{1}_{\{V_{r}^{N}=+1\}}dr + 2Q_{T}^{N,+,2}+ C.
\end{align*}
Hence tacking expectation in both side, we deduce that 
\begin{align}
\mathbb{E}\left(\sup_{0\leq s\leq T} H_{s}^{N}\right)&\leq 2\mathbb{E}\left(\sup_{0\leq s\leq T}  | {\tilde{M}}_{s}^{N} | \right)+ 2 \mathbb{E}\left(\sup_{0\leq s\leq T}|{\Psi}_{s}^{N}|\right)+ 2\mathbb{E}\left(Q_{T}^{N,+,2}\right)+C(1+T)  \nonumber\\
&\leq 2 \ \left[\mathbb{E}\left(\sup_{0\leq s\leq T} {\tilde{M}}_{s}^{N}\right)^2\right]^{\frac{1}{2}}+ 2\mathbb{E}\left(\sup_{0\leq s\leq T}|{\Psi}_{s}^{N}|\right)+ 2\mathbb{E}\left(Q_{T}^{N,+,2}\right)+ C(1+T) \nonumber.
\end{align}
This together with  Lemma \ref{QEGAK}, \eqref{Max}, Lemma \ref{ESPSISN}, \eqref{CROCHET}, Doob's $L^2$-inequality for martingales implies the result.
\end{proof}

This result is used to prove Lemma \ref{VTENDMI}. The rest is entirely similar to the supercritical case. Therefore, we obtain a similar convergence result.
\begin{theorem}
$H^{N}\Longrightarrow H$ in $\mathcal{C}([0,\infty))$, as $N \longrightarrow\infty$, where the process $H$ equals $2/\kappa$ multiplied by Brownian motion with drift $(\alpha-\beta)s / \kappa$, reflected above $0$.
\end{theorem}
\begin{remark}
The critical case cannot be treated as the subcritical case. In other words, in the case $\alpha=\beta$, we cannot choose $\Gamma=+\infty$, since $\mathbb{E} ({T}_{1}^N)$  would no longer be bounded (see Lemma \ref{ESPTUN}).
\end{remark}
\begin{remark}
From our convergence results, we can as in \cite{ba2012binary} (see also Theorem 3.1 in Delmas \cite{delmas2006height}) deduce the well known second Ray-Knight theorem, in the subcritical, critical and supercritical cases.
\end{remark}
$\mathbf{Ackownledgement}$.
The authors thank Fabienne Castell for a useful discussion concerning the proof of Lemma \ref{LEMLAMUN}.
\frenchspacing
\bibliographystyle{plain}

\end{document}